\documentclass[11pt, psamsfonts, reqno]{amsart}
\usepackage{amsmath}
\usepackage{amsthm}
\usepackage{amssymb}
\usepackage{amscd}
\usepackage{amsfonts}
\usepackage{amsbsy}
\usepackage{epsfig}

\newtheorem{theorem}{Theorem}
\newtheorem{remark}{Remark}
\newtheorem{proposition}{Proposition}
\newtheorem{lemma}{Lemma}
\newtheorem{corollary}{Corollary}

\newtheorem*{claim*}{Claim}

\def\H#1{{\bf #1}}

\def\ie{{\em i.e.,~}}

\newfont\bbf{msbm10 at 12pt}
\def\c{{\bf C}}
\def\eps{\varepsilon}
\def\phi{\varphi}
\def\R{{\mathbb R}}

\def\N{{\mathbb N}}

\def\E{{\mathcal E}}

\def\L{{\mathcal L}}
\def\H{{\mathcal H}}

\def\G{{\mathcal G}}

\def\P{{\mathcal P}}

\def\D{{\mathcal D}}
\def\M{{\mathcal M}}

\def\lev{\mbox{level}}
\def\plev{\mbox{pre-level}}
\def\slev{\mbox{\tiny level}}

\def\dm{\mathfrak D}

\def\tag#1{\hfill \qquad  #1}

\def\supp{\mbox{\rm supp}}

\def\orb{\mbox{\rm orb}}
\def\Crit{\mbox{\rm Crit}}
\def\smCrit{\mbox{\tiny Crit}}

\def\bd{\partial }
\def\le{\leqslant}
\def\ge{\geqslant}
\def\htop{h_{top}}
\def\laps{\mbox{laps}}

\def\rome{{\mathcal R}}
\def\smrome{{\scriptsize\mathcal R}}
\def\smromea{\mbox{\it \scriptsize rome}}

\newcommand{\st}{such that }

\newcommand{\cyc}{\mbox{cyc}}

\begin{document}
\bibliographystyle{plain}
\title[Equilibrium states for potentials with $\sup \phi - \inf \phi < \htop(f)$]
{Equilibrium states for interval maps:
potentials with \boldmath $\sup \phi - \inf \phi < \htop(f)$ \unboldmath}
\author{Henk Bruin, Mike Todd}
\thanks{
This research was supported by EPSRC grant GR/S91147/01.  MT was partially supported by  FCT grant SFRH/BPD/26521/2006 and CMUP}
\maketitle

\begin{abstract}
We study an inducing scheme approach for smooth
interval maps to prove existence and uniqueness
of equilibrium states for potentials $\phi$ with the `bounded range'
condition $\sup \phi - \inf \phi < \htop(f)$, first used by
Hofbauer and Keller  \cite{HKeq}.
We compare our results to Hofbauer and Keller's use of  Perron-Frobenius operators.  We demonstrate that this `bounded range' condition
on the potential is important even if the potential is H\"older continuous.
We also prove analyticity of the pressure in this context.
\end{abstract}

\section{Introduction}\label{sec:intro}
Thermodynamic formalism is concerned with
existence and uniqueness of measures $\mu_\phi$ that
maximise the {\em free energy}, \ie the sum of the entropy and
the integral over the potential.  In other words
\[
h_{\mu_\phi}(f) + \int_X \phi~d\mu_\phi  = P(\phi) := \sup_{\nu
\in \M_{erg}} \left\{  h_{\nu}(f) + \int_X \phi \
d\nu:-\int_X\phi~d\nu<\infty \right\}
\]
where $\M_{erg}$ is the set of all ergodic $f$-invariant Borel
probability measures. Such measures are called {\em equilibrium
states}, and $P(\phi)$ is the {\em pressure}. This theory
was developed by Sinai, Ruelle and Bowen \cite{Sinai, Ruelle,  Bowen} in
the context of H\"older potentials on hyperbolic dynamical
systems, and has been applied to Axiom A systems, Anosov
diffeomorphisms and other systems too, see e.g.\
\cite{Baladi,Kbook} for more recent expositions.

In this paper we are interested in smooth interval maps
$f:I \to I$ with a finite number of critical points.
More precisely, $\H$ will be the collection of \emph{topologically
mixing} (\ie for each $n \ge 1$, $f^n$ has a dense orbit)
$C^2$  maps on the interval (or circle) such that all its periodic
points are hyperbolically repelling and all its critical points are non-flat.
The existence of critical points prevents such maps from being
uniformly hyperbolic for the `natural' potential $\phi = -\log|Df|$.

Inducing schemes where used in \cite{PS} to regain hyperbolicity and
prove the existences of equilibrium states for $-t \log|Df|$
for a large interval of $t$, but very specific Collet-Eckmann unimodal maps $f$.
In \cite{BTeqnat} we investigated $-t \log|Df|$ with $t$ close to $1$
for multimodal maps
whose derivatives critical orbits satisfy only polynomial growth.
Combining inducing schemes
with ideas of so-called Hofbauer towers and infinite state Markov chains
(as presented by Sarig \cite{Sathesis,Saphase,SaBIP}),
we proved the existence and uniqueness of equilibrium states
within the class
\[
\M_+ = \left\{ \mu \in \M_{erg} : 
\lambda(\mu) > 0, \supp(\mu) \not\subset
\orb(\Crit) \right\}.
\]
where $\lambda(\mu) = \int \log |Df| d\mu$ is the Lyapunov exponent
of $\mu$.
In fact the assumptions that we make on the potentials in this paper ensure that any equilibrium state must lie in this class, and hence it is no restriction to only consider measures there.

\begin{remark}\label{rem:existence}
Note that the function $\mu \mapsto h_\mu(f)$ is upper semicontinuous,
cf.\ \cite[Lemma 2.3]{BrKel}.  Hence, if the  potential is upper semicontinuous, then the free energy
map $\mu \mapsto h_\mu(f) + \int \phi \ d\mu$ is upper
semicontinuous too.  As $\M_{erg}$ is compact in the weak topology, this
gives the existence of equilibrium states, but not uniqueness.
\end{remark}

In this work we want to use inducing schemes to prove
existence and uniqueness of equilibrium states for
``general'' potentials. In this area, there are many results,
in particular several papers by Hofbauer and Keller
\cite{Hnonuni,Hpwise,HKeq} from the late 1970s. These results
were inspired by Bowen's exposition \cite{Bowen}
for hyperbolic dynamical systems,
and investigate what happens when hyperbolicity fails.
Their main tool was the Perron-Frobenius operator, which
even for non-uniformly expanding interval maps continues to have a quasi-compact structure for many potentials.
In this paper we focus on what can be proved for these problems using inducing techniques.  We then apply Sarig's theory of countable Markov shifts.  (A related application of that theory for multidimensional piecewise expanding maps can be found in \cite{BuSar}.)  In \cite{HKeq} two main sets of results are given,
based on different regularity conditions for the potential;
we will present them briefly in Sections~\ref{subsec:BV} and
\ref{subsec:var}.  At the same time we set out some definitions which will be used throughout the paper. In Section~\ref{subsec:main} we present our main results.

\subsection{Potentials in \boldmath $BV$ \unboldmath}
\label{subsec:BV}

Given a function $\phi:I \to \R$, we define the
semi-norm $\|\cdot\|_{BV}$ as
$$
\|\phi\|_{BV}:=\sup_{N \in \N}\
\sup_{0=a_0<\cdots<a_N=1}\sum_{k=0}^{N-1}|\phi(a_{k+1})-\phi(a_k)|.
$$
We say that $\phi\in BV$ if $\|\phi\|_{BV}<\infty$.

The following result is proved by Hofbauer and Keller in \cite{HKeq}.

\begin{theorem}[Hofbauer  and  Keller]
\label{thm:HK}
Let $f \in \H$ and $\phi \in BV$. If
\begin{equation}\label{eq:bdd range}
\sup \phi - \inf \phi < \htop(f),
\end{equation}
then there exists an equilibrium state for $\phi$.
Moreover, the transfer operator defined by
\begin{equation*}\label{eq:L}
\L_\phi g(x) := \sum_{y\in f^{-1}(x)} e^{\phi(y)} g(y)
\end{equation*}
is quasi-compact.
\end{theorem}

Condition \eqref{eq:bdd range} stipulates that $\phi$ does not vary too much; similar conditions have been used by e.g.\ Denker and Urba\'nski \cite{DUeq} for rational maps on the Riemann sphere, and by Oliveira \cite{O} for higher dimensional maps without critical points.
We next state a similar result to Theorem~\ref{thm:HK} from \cite{DKU,Pac}.
Paccaut \cite{Pac} also gives
many interesting statistical properties for the equilibrium states.

\begin{theorem}[Paccaut]
Suppose that $\phi$ satisfies
\begin{itemize}
\item[(a)] $\exp(\phi)\in BV$;
\item[(b)] $\sum_{n=1}^\infty \sup_{C\in \P_n}\|\phi|_C\|_{BV}<\infty$;
\item[(c)] $\sup\phi<P(\phi)$.
\end{itemize}
Then there exists a unique equilibrium state $\mu_\phi$ for $\phi$.
\label{thm:Pac}
\end{theorem}

Note that condition (b) on $\phi$ is stronger than the condition $\phi\in BV$, used in Theorem~\ref{thm:HK}.  It is also stronger than that in our results in Section~\ref{subsec:main}.  However, \eqref{eq:bdd range} implies condition (c).  This follows since assuming \eqref{eq:bdd range},
the measure of maximal entropy  $\mu_{\htop(f)}$ gives
$$
P(\phi) \ge \htop(f)+\int \phi~d\mu_{\htop(f)}\ge \htop(f) + \inf\phi > \sup\phi.
$$
Condition (c) implies that any equilibrium state $\mu$ must have
$h_\mu(f) \ge P(\phi) - \sup \phi > 0$.
Similarly, supposing \eqref{eq:bdd range},
and using Ruelle's inequality on Lyapunov exponents (\ie $h_\mu(f) \le \lambda(\mu)$, see \cite{Ruelleineq}),
equilibrium states $\mu$ satisfy
\begin{eqnarray}\label{eq:lmuhmu}
\lambda(\mu) &\ge& h_\mu(f) =P(\phi)-\int\phi~d\mu \nonumber \\
&\ge&
\htop(f) + \int\phi~d\mu_{\htop(f)}-\sup\phi\ge \htop(f)-(\sup\phi-\inf\phi)>0.
\end{eqnarray}
Hence $P_+(\phi) := \sup_{\mu \in \M_+} \{ h_\mu(f) + \int \phi \ d\mu \}
=P(\phi)$, unless the equilibrium state is supported on $\orb(\Crit)$.

\subsection{Potentials with Summable Variations.}
\label{subsec:var}
The results that we want to present rely on a different approach
to variation to that above, which is closer to symbolic dynamics.
Let $\P_1$ be the partition of $I$ into maximal interval of
monotonicity (the \emph{branch partition}) and write
$\P_n = \bigvee_{i=0}^{n-1} f^{-i}(\P_1)$.
With respect to this partition
we define that {\em $n$-th variation}
\begin{equation*}
V_n(\phi) := \sup_{\c_n \in \P_n} \sup_{x,y \in \c_n} |\phi(x) -
\phi(y)|,
\end{equation*}
In this context the following was proved in \cite{HKeq}.
\begin{theorem}[Hofbauer  and  Keller]
\label{thm:HKsum}
Let $f \in \H$ be $C^3$ and let $\phi$ be a potential so that
\begin{enumerate}
\item[(i)]
it has summable variations, \ie
$\sum_n V_n(\phi) < \infty$;
\item[(ii)] the following specification-like property holds:
for every $x \in I$, there is $k$ and an increasing sequence $\{n_i\}_i$
such that
\[
\cup_{j=1}^k f^{n_i+j}(\c_{n_i}[x]) = I,
\]
where $\c_{n_i}[x] \in \P_{n_i}$ is the $n_i$-cylinder containing $x$.
\end{enumerate}
Then there exists an equilibrium state for $\phi$
and the transfer operator $\L_\phi$ is quasi-compact.
\end{theorem}
Property (ii) above is not automatic for interval maps, and it is stronger than the standard specification property which holds for all topologically
transitive interval maps, see \cite{Blokh} and \cite{Buspeci}.
For instance, the Fibonacci unimodal map, or more generally, every
map with a persistently recurrent critical point (see e.g.\ \cite{BruinTAMS}) fails this condition.
In \cite{DKU}, Denker et al. replace the conditions of
Theorem~\ref{thm:HKsum} to (i) $P(\phi) > \sup \phi$ and (ii)
$\sup_n \beta_n(\phi) < \infty$, where $\beta_n$ is defined in \eqref{eq:beta}.

Notice that the set of potentials with summable variations and the set $BV$ have non-empty intersection, but neither is contained in the other, as the following examples demonstrate.

{\bf Example 1:}
Let $f(x) = 2x \pmod 1$ on $[0,1]$ be the doubling map.
Clearly, the $n$-cylinders of $f$ are dyadic intervals of length $2^{-n}$.
The potential function
\[
\phi(x) := \left\{ \begin{array}{lll}
0 & \mbox{ if } x = 0;\\
\frac{-1}{\log x} & \mbox{ if } x \in (0, \frac12);\\
\frac1{\log 2}& \mbox{ if } x \in [\frac12,1],
\end{array} \right.
\]
is increasing and bounded, and has $\|\phi\|_{BV}=\frac1{\log 2}$. However,
$V_n(\phi) \ge \frac1{n\log 2}$, because $\phi(2^{-n}) - \phi(0) = \frac1{n\log 2}$.
So $\sum_n V_n(\phi)$ diverges. Note that $\phi$ is not H\"older either.

{\bf Example 2:}
For $f$ as in Example 1, the potential function
\[
\psi(x) := \sum_{n \ge 1} \psi_n(x), \hbox{ where } \psi_n(x):= 4^{-n} \sin(4^{n+1}\pi x)\cdot 1_{\left[\frac1{2^n},\frac1{2^{n-1}}\right]}(x) \]
has $\|\psi\|_{BV}= \sum_n\|\psi_n\|_{BV} = \infty$ since $\|\psi_n\|_{BV}=2$.
But $V_n(\psi) \le 4 \cdot 2^{-n}$, so it has summable variations.  Note that this function is Lipschitz.

\subsection{Lifting Potentials to Inducing Schemes}\label{subsec:inducelifting}

An inducing scheme $(X,F,\tau)$ over $(I,f)$ consists of an
interval $X \subset I$ containing a (countable) collection of disjoint
subintervals $X_i$, and inducing time $\tau:X \to \N$ such that
$\tau_i := \tau|_{X_i}$ is constant and $F|_{X_i}  := f^{\tau_i}|_{X_i}$
is monotone onto $X$. If $\mu_F$ is an $F$-invariant measure,
and $\int_X \tau d\mu_F < \infty$, then $\mu_F$ can be projected to an $f$-invariant measure $\mu$ as in formula \eqref{eq:lift} below.
Any measure $\mu$ that can be obtained this way is called compatible to
the inducing scheme.
See Section~\ref{subsec:induce} the precise definitions.

Proposition~\ref{prop:induce} below
gives a general way of constructing
inducing schemes, which we will apply throughout the paper.
In Section~\ref{subsec:Hofbauer tower}, we explain the procedure
of lifting measures $\mu$ to
Hofbauer tower $(\hat I, \hat f)$, which is behind the construction
in this proposition. The full proof of Proposition~\ref{prop:induce}
is given in \cite[Theorem 3 and Lemma 2]{BTeqnat}.

\begin{proposition}\label{prop:induce}
If $\mu \in \M_+$ then it is compatible to some induced system $(X,F,\tau)$
that corresponds to a first return map to a set $\hat X$ on the Hofbauer tower, where $\hat\mu(\hat X)>0$.
So $\frac{1}{\hat\mu(\hat X)} \int_{\hat X} \tau\ d\hat\mu < \infty$, and in addition, we can take $X \in \P_n$
for some $n$.

Conversely, if an inducing scheme $(X,F,\tau)$ has a non-atomic
$F$-invariant measure $\mu_F$ such that
$\int \tau \ d\mu_F < \infty$,
then it projects to an $f$-invariant measure $\mu \in \M_+$.
\end{proposition}

Given a potential $\phi$, the lifted potential $\Phi$
for inducing scheme $(X,F,\tau)$ is
given by $\Phi(x) := \sum_{k=0}^{\tau(x)-1} \phi \circ f^k(x)$.
If
\begin{equation}\label{SVI}
\sum_nV_n(\Phi) < \infty \tag{SVI},
\end{equation}
then we say that $\phi$ satisfies the summable variations for induced
potential condition, with respect to this inducing scheme.
Lemmas~\ref{lem:sum induced} and \ref{lem:Holder} give
general conditions on $\phi$ and/or the inducing scheme that imply (SVI).

\subsection{Main Results}\label{subsec:main}
After these preparation we can state our main results on the existence
and uniqueness of equilibrium states, and analyticity of the pressure function.
The existence of equilibrium states
in $\M_{erg}$ often follows by Remark~\ref{rem:existence},
but the following theorem gives conditions for uniqueness of equilibrium
states in $\M_+$.

\begin{theorem}
Let $f\in \H$ and $\phi$ be a potential \st $\sup \phi - \inf \phi < \htop(f)$ and $V_n(\phi) \to 0$.  If the induced potentials corresponding to the inducing schemes given by Proposition~\ref{prop:induce} satisfies \eqref{SVI}, then
\begin{itemize}
\item[(a)] there exists a unique equilibrium state $\mu_\phi$;
\item[(b)] $\mu_\phi$ is compatible to an induced system with inducing time
such that the tails $\mu_\Psi(\{ \tau > n\})$ decrease exponentially.
(Here $\mu_\Psi$ is the equilibrium state of the induced potential
$\Psi(x) = \sum_{k=j}^{\tau(x)-1} \psi \circ f^j(x)$ of $\psi := \phi - P(\phi)$.)
\end{itemize}
\label{thm:exist for bdd range}
\end{theorem}

Note that $V_n(\phi) \to 0$ implies that $\phi$ can only have discontinuities
at precritical points.

\begin{remark}\label{rm:for bdd range}
If the tails $\mu_{\Psi}(\{\tau > n\})$ decrease at certain rates,
then one can
deduce many statistical properties of the equilibrium state.
For instance, exponential decay of correlations follows from
exponential tails, see \cite{Y}, but for the Central Limit Theorem,
Invariance Principles, e.g.\ \cite{MN1} and large deviations \cite{MN2},
already polynomial tail behaviour suffices.
\end{remark}

Instead of a single potential, thermodynamic formalism makes use of
families $t \phi$ of potentials.
The occurrence of phase transitions is related to the smoothness
of the pressure
function $t \mapsto P(t \phi)$.
Using the technique in \cite{BTeqnat} we derive

\begin{theorem}
Let $f\in \H$ and $\phi$ as in Theorem~\ref{thm:exist for bdd range}.
Then the map $t \mapsto P(-t \, \phi)$ is analytic for $t$ in a
neighbourhood of $[-1,1]$. \label{thm:ana gen}
\end{theorem}

We will not supply a proof of the above theorem, since it follows rather
easily from \cite[Theorem 5]{BTeqnat}.  We will focus our attention on the
following related theorem dealing with the potential $-t \log |Df|$.
This potential is unbounded, except for $t = 0$.  We conclude that
$t \mapsto P(-t \log|Df|)$ is analytic near $t = 0$, which is
somewhat surprising as we do not require any of the summability conditions of the critical orbits of $f$ used in \cite{BTeqnat}.

\begin{theorem}
Let $f\in \H$.  There exist $t_1<0<t_2$ so that the map $t \mapsto P(-t\log|Df|)$ is analytic for $t\in (t_1,t_2)$.  In fact, for $t\in (t_1,t_2)$ there exists a unique equilibrium state with respect to the potential $-t\log|Df|$.
\label{thm:ana nat}
\end{theorem}

We next make a detailed study of an example
by Hofbauer and Keller \cite[pp32-33]{HKeq} which applies ideas from \cite{Hnonuni}.
They used it to show the importance of the condition \eqref{eq:bdd range}
for the quasi-compactness of the transfer operator.
We use the example to test the restrictions of the inducing scheme methods,
and we also show that \eqref{eq:bdd range} cannot simply be replaced by
H\"older continuity of the potential by proving the following proposition,
cf. \cite{Saphase}.

\begin{proposition} \label{prop:MP}
For $\alpha\in (0,1)$, consider the Manneville-Pomeau map
$f_\alpha:x \mapsto x+ x^{1+\alpha} \pmod 1$.
For any $b<-\log 2$,  there exists a H\"older potential with
$\sup \phi - \inf \phi = |b|$ and which  has the form
$\phi(x) = -2\alpha x^{\alpha}$ for $x$ close to $0$,
which has no equilibrium state accessible from an inducing scheme given by Proposition~\ref{prop:induce}.
\end{proposition}

The remainder of this paper is organised as follows.
In Section~\ref{sec:inducing} we set out our main tools for generating inducing schemes and applying the theory of thermodynamic formalism.
Section~\ref{sec:tail} contains the
tail estimates of inducing schemes we use.
In Section~\ref{sec:exist for bdd range} we prove our main theorem on existence and uniqueness of equilibrium states.
In Section~\ref{sec:ana} we show that a consequence of our results is an analyticity result for the pressure, with respect to the kind of potentials considered in \cite{BTeqnat}. In Section~\ref{sec:examples} we give examples, including that in Proposition~\ref{prop:MP}, to show where these techniques break down.
Finally in Section~\ref{sec:recurrence} we discuss the recurrence implied by
compactness of the transfer operator, and we present conditions
implying the recurrence of the potential $\phi$.

{\em Acknowledgements:} We would like to thank Ian Melbourne,
Beno\^{\i}t Saussol, Godofredo Iommi, Sebastian van Strien and Neil Dobbs for fruitful discussions.  We would also like to thank the LMS for funding the visit of Saussol.  HB would like to thank CMUP for its hospitality. We also thank the referee for careful
reading and constructive comments.

\section{Equilibrium States via Inducing}
\label{sec:inducing}

\subsection{Inducing Schemes}\label{subsec:induce}
As in \cite{BTeqnat} we want to construct equilibrium state via inducing schemes.
We say that $(X,F,\tau)$ is an {\em inducing scheme} over $(I,f)$
if
\begin{itemize}
\item $X$ is an interval\footnote{Due to our assumption that $f$ is topological mixing, we can always find a single interval to induce on, but similar
theory works for $X$ a finite union of intervals.}
containing a (countable)
collection of disjoint intervals $X_i$ \st $F$ maps each $X_i$
homeomorphically onto $X$.
\item $F|_{X_i} = f^{\tau_i}$ for some $\tau_i \in \N := \{ 1,2,3 \dots \}$.
\end{itemize}
The function $\tau:\cup_i X_i \to \N$ defined by $\tau(x) =
\tau_i$ if $x \in X_i$, is called the {\em inducing time}.
It may happen that $\tau(x)$ is the first return time of $x$ to $X$, but
that is certainly not the general case.
Given an inducing scheme $(X,F, \tau)$, we say that a measure
$\mu_F$ is a \emph{lift} of $\mu$ if for all $\mu$-measurable subsets
$A\subset I$,
\begin{equation} \mu(A) = \frac1{\Lambda_{F,\mu}} \sum_i \sum_{k = 0}^{\tau_i-1}
\mu_F( X_i \cap f^{-k}(A)) \quad \mbox{ for } \quad
\Lambda_{F,\mu} := \int_X \tau \ d\mu_F. \label{eq:lift}
\end{equation}
Conversely, given a measure $\mu_F$ for $(X,F)$, we say that
$\mu_F$ \emph{projects} to $\mu$ if \eqref{eq:lift} holds.

Not every inducing scheme is relevant to every invariant measure.
Let $X^\infty = \cap_n F^{-n}(\cup_i X_i)$ is the set of
points on which all iterates of $F$ are defined. We call a measure
$\mu$ \emph{compatible} with the inducing scheme if
\begin{itemize}
\item $\mu(X)> 0$ and $\mu(X \setminus X^{\infty}) = 0$, and
\item there exists a measure $\mu_F$ which projects to $\mu$ by
\eqref{eq:lift}, and in particular $\Lambda_{F,\mu} < \infty$.
\end{itemize}

\subsection{The Hofbauer Tower}\label{subsec:Hofbauer tower}
Let $\P_n$ be the branch partition for $f^n$.  The canonical Markov extension (commonly called {\em Hofbauer
tower}) is a disjoint union of subintervals $D = f^n(\c_n)$,
$\c_n \in \P_n$, called {\em domains}. Let $\D$ be the collection of all such domains. For
completeness, let $\P_0$ denote the partition of $I$ consisting of
the single set $I$, and call $D_0 = f^0(I)$ the {\em base} of the
Hofbauer tower. Then
\[
\hat I = \sqcup_{n \ge 0} \sqcup_{\c_n \in \P_n} f^n(\c_n) / \sim,
\]
where $f^n(\c_n) \sim f^m(\c_m)$ if they represent the same
interval. Let $\pi: \hat I \to I$ be the inclusion map. Points
$\hat x\in \hat I$ can be written as $(x,D)$ if $D\in \D$ is the domain
that $\hat x$ belongs to and $x = \pi(\hat x)$. The map $\hat
f:\hat I \to \hat I$ is defined as
\[
\hat f(\hat x) = \hat f(x,D) = (f(x), D')
\]
if there are cylinder sets $\c_n \supset \c_{n+1}$ \st $x \in
f^n(\c_{n+1}) \subset f^n(\c_n) = D$ and $D' = f^{n+1}(\c_{n+1})$.
In this case, we write $D \to D'$, giving $(\D, \to)$ the
structure of a directed graph. It is easy to check that there is a
one-to-one correspondence between cylinder sets $\c_n \in \P_n$
and $n$-paths $D_0 \to \dots \to D_n$ starting at the base of the
Hofbauer tower and ending at some {\em terminal domain}
$D_n$.  If $R$ is the length of the shortest path from the base to
$D_n$, then the {\em level of $D_n$} is $\lev(D_n) = R$. Let $\hat
I_R = \sqcup_{\slev(D) \le R} D$.

Several of our arguments rely on the fact that the
``top'' of the infinite graph
$(\D,\to)$ generates arbitrarily small entropy.  These ideas go back to Keller \cite{Kellift}, see also \cite{Bumulti}.
It is also worth noting that the main information is contained in a
single transitive part of $\hat I$.

\begin{lemma}\label{lem:trans_subgraph}
If $I$ is a finite union of intervals, and the multimodal map $f:I
\to I$ is transitive, then there is a closed primitive subgraph
$(\E, \to)$ of $(\D, \to)$ containing a dense $\hat f$-orbit and
\st $I = \pi(\cup_{D \in \E} D)$.
\end{lemma}

We denote the transitive part of the Hofbauer tower by $\hat I_{\mbox{\tiny trans}}$.
For details of the proof see \cite[Lemma 1]{BTeqnat}.

Let $i:I \to D_0$ be the trivial bijection (inclusion)
\st $i^{-1} = \pi|_{D_0}$. Given a probability measure $\mu$, let
$\hat \mu_0 := \mu \circ i^{-1}$, and
\begin{equation}
\hat{\mu}_n := \frac{1}{n}\sum_{k=0}^{n-1} \hat \mu_0
\circ\hat{f}^{-k}. \label{eqn:mulift}
\end{equation}
We say that $\mu$ is {\em liftable} to $(\hat I, \hat f)$ if there
exists a vague accumulation point $\hat \mu$ of the sequence
$\{\hat\mu_n\}_n$ with $\hat\mu\not\equiv 0$, see \cite{Kellift}. The following theorem is essentially proved there, see \cite{BrKel} for more details.

\begin{theorem}
Suppose that $\mu\in \M_+$.  Then $\hat\mu$ is an $\hat f$-invariant
probability measure on $\hat I$, and $\hat\mu\circ\pi^{-1} = \mu$.

Conversely, if $\hat\mu$ is $\hat f$-invariant and non-atomic,
then $\lambda(\hat\mu) > 0$.
\label{thm:Kell}\end{theorem}

The strategy followed in \cite{BTeqnat} is to take the first return map to appropriate set
in the Hofbauer tower of $(I,f)$ and to use the same inducing time for the projected partition on the interval.
Saying that an induced system $(X,F,\tau)$  corresponds to a first return map
$(\hat X, \hat F, \tau)$ on the Hofbauer tower means that
if $\hat x \in \hat X \subset \hat I$, then $\tau \circ \pi$
is the first return time of $\hat x$ under $\hat f$ to $\hat X$.

\subsection{Pressure and Recurrence}\label{subsec:pressure}

A topological, \ie measure independent, way to define
pressure was presented in \cite{walters}; with respect to
the branch partition $\P_1$, it is defined as
\[
P_{top}(\phi) := \lim_{n\to\infty} \frac1n \log
\sum_{\c_n \in \P_n} \sup_{x \in \c_n} e^{\phi_n(x)},
\]
where $\phi_n(x) := \sum_{k=0}^{n-1} \phi \circ f^k(x)$.
We say that the Variational Principle holds if $P(\phi) = P_{top}(\phi)$.
If $\phi$ has sufficiently controlled distortion, then
the sum of $\sup_{x \in \c_n} e^{\phi_n(x)}$ over
all $n$-cylinders can be replaced by
the sum of $e^{\phi_n(x)}$ over all $n$-periodic points, and thus we arrive
at the {\em Gurevich pressure} w.r.t. cylinder set $\c \in \P_1$.
\begin{equation*}\label{eq:Gur}
P_G(\phi) := \limsup_{n \to \infty} \frac1n \log Z_n(\phi, \c)
\quad \mbox{ for } \quad
Z_n(\phi, \c) := \sum_{f^n x = x}e^{\phi_n(x)}1_{\c}(x).
\end{equation*}
If $(I,f)$ is topologically mixing and
\begin{equation}\label{eq:beta}
\beta_n(\phi) := \sup_{\c_n\in \P_n} \sup_{x,y \in
\c_n}|\phi_n(x)-\phi_n(y)|=o(n),
\end{equation}
then
$P_G(\phi)$ is independent of the choice of $\c \in \P_1$, as was shown in \cite{FFY}.

Since the branch partition is finite, potentials with bounded variations are
bounded, and hence their Gurevich pressure is finite.
If $\phi$ is unbounded above (whence $P_{top}(\phi) = \infty$) or the number of
$1$-cylinders is infinite
(as may be the case for induced maps $F$ and induced potential $\Phi$), Gurevich pressure proves its usefulness.

Suppose that $(I,f,\phi)$ is topologically mixing.  For every $\c
\in \P_1$ and $n\ge 1$, recall that we defined
\[
Z_n(\phi, \c):=\sum_{f^n x = x}e^{\phi_n(x)}1_{\c}(x).
\]
Let
\[
Z_n^*(\phi, \c) := {\sum_{\stackrel{f^nx=x,}{f^kx \notin \c \ \mbox{\tiny for}\ 0< k < n}} e^{\phi_n(x)}} 1_{\c}(x).
\]
The potential $\phi$ is said to be {\em recurrent} if\footnote{The convergence of this series is independent of the cylinder set $\c$, so we suppress it in the notation.}
\begin{equation}\label{eq:recurrence}
\sum_n \lambda^{-n} Z_n(\phi) = \infty \mbox{ for } \lambda = \exp
P_G(\phi).
\end{equation}
Moreover, $\phi$ is called \emph{positive recurrent} if it is
recurrent and  $\sum_n n\lambda^{-n}Z^*_n(\phi) < \infty$.

In some cases we will use the quantity
\begin{equation}\label{eq:Z0}
Z_0(\phi):=\sum_{\c\in \P_1} \sup_{x\in \c} e^{\phi(x)}.
\end{equation}
Proposition 1 of \cite{Sathesis} implies that if $\phi$ has summable variations then for any $\c$,  $Z_n(\phi,\c)=O(Z_0(\phi)^n)$.  Hence  $Z_0(\phi)<\infty$ implies $P_G(\phi)<\infty$.

Although we do not assume that the potential $\phi$ has summable variations,
it is important that the induced potential $\Phi$
has summable variations, as we want to apply the following
result which collects the main theorems of \cite{SaBIP}.
We give a simplified version of the original result since we
assume that each branch of the induced system $(X,F)$ is onto $X$.
We refer to such a system as a \emph{full shift}.

\begin{theorem} If $(X,F,\Phi)$ is a full shift and $\sum_{n\ge 1}V_n(\Phi)<\infty$, then
$\Phi$ has an invariant Gibbs measure if and only if
$P_G(\Phi)<\infty$.  Moreover the Gibbs measure $\mu_\Phi$ has the
following properties.
\begin{itemize}
\item[(a)] If $h_{\mu_\Phi}(F) < \infty$ or $-\int \Phi d\mu_\Phi
< \infty$ then $\mu_\Phi$ is the unique equilibrium state (in
particular, $P(\Phi) = h_{\mu_\Phi}(F) + \int_X \Phi~d\mu_\Phi$);
\item[(b)] The Variational Principle holds, \ie
$P_G(\Phi)=P(\Phi)$.
\end{itemize}
\label{thm:BIP}
\end{theorem}
Note that an $F$-invariant measure $\mu$ is a
\emph{Gibbs measure} w.r.t. potential $\Phi$
if there is $K \ge 1$ such that
for every $n \ge 1$, every $n$-cylinder set $\c_n$ and every $x \in \c_n$
\[
\frac1K \le \frac{ \mu(\c_n) }{ e^{\Phi_n(x) - nP_G(\Phi)} } \le K.
\]

Using this theory, the following was proved in \cite{BTeqnat}.
\begin{proposition}\label{prop:all_indu}
Suppose that $\psi$ is a potential with $P_G(\psi)=0$.  Let $\hat X$ be the set used Proposition~\ref{prop:induce} to construct the corresponding inducing scheme
$(X,F,\tau)$.  Suppose that the lifted potential $\Psi$ has $P_G(\Psi)<\infty$ and $\sum_{n\ge 1}V_n(\Psi)<\infty$.

Consider the assumptions:
\begin{itemize}
\item[(a)] $\sum_i\tau_i e^{\Psi_i}<\infty$ for
$\Psi_i := \sup_{x \in X_i} \Psi(x)$;
\item[(b)] there exists an equilibrium state
$\mu\in \M_+$ compatible with $(X,F,\tau)$;
\item[(c)]
there exist a sequence
$\{\eps_n\}_n\subset \R^-$ with
$\eps_n \to 0$ and measures
$\{\mu_n\}_n\subset \M_+$ \st every $\mu_n$ is compatible with $(X,F,\tau)$,
$h_{\mu_n}(f)+\int\psi~d\mu_n \ge \eps_n$ and
$P_G(\Psi_{\eps_n})<\infty$ for all $n$;
\end{itemize}
If any of the following combinations of assumptions holds:
\[
\left\{
\begin{array}{ll}
1. & \mbox{(a) and (b)}; \\
2. & \mbox{(a) and (c)};
\end{array} \right.
\]
then there is a unique equilibrium state $\mu$ for $(I,f,\psi)$
among measures $\mu \in \M_+$ with $\hat\mu(\hat X)>0$. Moreover, $\mu$ is obtained by projecting the equilibrium state $\mu_\Psi$ of the inducing scheme and we have $P_G(\Psi)=0$.
\end{proposition}

In the remaining part of this section, we give some technical results which connect different ways of computing pressure and Gurevich pressure.

We use the following theorem of \cite{FFY} to show the connection between $P_G(\hat\phi)$ and $P_+(\phi)$.
\begin{theorem}
If $(\Omega,S)$ be a transitive Markov shift and $\psi:\Omega\to \R$
is a continuous function satisfying $\beta_n(\psi)=o(n)$ then $P_G(\psi)=P(\psi)$.
\label{thm:FFY}
\end{theorem}

\begin{corollary}
If $\beta_n(\hat\phi) = o(n)$, and $\hat\phi$ is continuous in the symbolic metric on $(\hat I, \hat f)$ then $P_G(\hat \phi)=P_+(\phi)$. \label{cor:var prin}
\end{corollary}

\begin{proof}
We show that the system $(\hat I_{\mbox{\tiny trans}}, \hat f,\hat \phi)$
satisfies the conditions of Theorem~\ref{thm:FFY}, where
$\hat I_{\mbox{\tiny trans}}$ is given below Lemma~\ref{lem:trans_subgraph}.
For $\hat x,\hat y \in \hat P$ with $\hat P \in \hat \P_n$, we have
 $|\hat\phi_n(\hat x) - \hat\phi_n(\hat y)|= o(n)$,
and Theorem~\ref{thm:FFY} implies $P_G(\hat\phi)=P(\hat\phi)$.

It remains to show that $P(\hat \phi) = P_+(\phi)$.  By Theorem~\ref{thm:Kell}, any measure in $\M_+$ lifts to $\hat I$.  We also know that a countable-to-one factor map preserves entropy, provided the Borel sets are preserved by lifting, see \cite{DS}.  For similar arguments, see \cite{Bumulti}.
Suppose that $\{\hat\mu_n\}_n$ is a sequence of $\hat f$-invariant
measures \st $h_{\hat\mu_n}(f)+\int\hat\phi~d\hat\mu_n \to P(\hat\phi)$
as $n \to \infty$.
Then for the projections $\mu_n=\hat\mu_n\circ\pi^{-1}$, $h_{\mu_n}(f)+\int\phi~d\mu_n \to P(\hat\phi)$ also.
So $P_+(\phi) \ge P(\hat\phi)$.
On the other hand,  let $\{\mu_n\}_n \subset \M_+$ be a sequence of
measures \st $h_{\mu_n}(f)+\int\phi~d\mu_n \to P_+(\phi)$ as $n \to \infty$.
Lifting these measures using Theorem~\ref{thm:Kell},
we get $h_{\hat\mu_n}(f)+\int\hat\phi~d\hat\mu_n \to P_+(\phi)$, so $P_+(\phi) \le P(\hat\phi)$ as required. \end{proof}

We next show that Gurevich pressure can be computed from cylinders of any order.

\begin{lemma}
Let $(\Omega, f)$ be a topologically mixing Markov shift.
If $\phi:\Omega \to \R$ satisfies $\beta_n(\phi)=o(n)$,
then $P_G(\phi,\c)=P_G(\phi,\c')$ for any two cylinders $\c,\c'$ of any order.
\label{lem:small cyl}
\end{lemma}

\begin{proof} Denote the Markov partition of $\hat I$ into domains $D$
by $\D$.
Take $D, D' \in \D$ such that $\c \subset D$ and $\c' \subset D'$.
By transitivity, there is a $k$-path $\c \subset D \to \dots \to D'$ and a $k'$-path $\c' \subset D' \to \dots \to D$.
Then for every $n$-periodic point $x \in \c$, there is a
point $x' \in \c'$ such that $f^{k'}(x') \in \c_n[x]$, the $n$-cylinder
containing $x$. Therefore $f^{k'+n}(x') \in \c$ and  $f^{k'+n+k}(x') \in x'$.
It follows that
$e^{\phi_{n+k+k'}(x')} \le e^{\beta_n + (k+k') \sup \phi} e^{\phi_n(x)}$,
whence
\[
Z_n(\phi,\c) \ge e^{-\beta_n - (k+k')\sup \phi} Z_{n+k+k'}( \phi, \c').
\]
Therefore, using $\beta_n = o(n)$, we obtain for the exponential growth rate
$P_G(\phi,\c) \ge \lim_n \frac{\beta_n}{n} + P_G(\phi, \c') = P_G(\phi, \c')$.
Reversing the roles of $\c$ and $\c'$ yields $P_G(\phi,\c) = P_G(\phi, \c')$.
\end{proof}

\subsection{Summable Variations for the Inducing Scheme (SVI)}
\label{subsec:SVI}

In this section we give conditions on $\phi$ and under which (SVI) holds
for the inducing scheme.

\begin{lemma}
\begin{itemize}
\item[(a)] If
\begin{equation*}\label{eq:sumn}
\sum_n n V_n(\phi) < \infty;
\end{equation*}
then \eqref{SVI} holds with respect to any inducing scheme.
\item[(b)]
Let $\phi$ be $\alpha$-H\"older continuous and let $(X,F,\tau)$
be an inducing scheme obtained from Proposition~\ref{prop:induce}
that satisfies
\begin{equation}\label{eq:sumalpha}
\sup_i \sum_{k=0}^{\tau_i-1}|f^k(X_i)|^\alpha<\infty,
\end{equation}
Then \eqref{SVI} holds w.r.t. that inducing scheme.
\end{itemize}
\label{lem:sum induced}\end{lemma}

\begin{proof}
To prove (a), we apply \cite[Lemma 3, Part 1]{Sathesis}.  Note that the results in the chapter of \cite{Sathesis} containing this result  are valid if $(X,F,\tau)$ is a first return map, which is not true for our case. However, from Proposition~\ref{prop:induce}, we constructed $(X,F,\tau)$ to be isomorphic to a
first return map on the Hofbauer tower, with potential $\hat \phi
= \phi \circ \pi$. Since $\Phi(x) = \sum_{k = 0}^{\tau(x)-1} \phi
\circ f^k(x) = \sum_{k = 0}^{\tau(x)-1} \hat \phi \circ \hat
f^k(\hat x)$ for each $\hat x \in \pi^{-1}(x)$, both the original
system and the lift to the Hofbauer tower lead to the same induced
potential. Therefore \cite[Lemma 3, Part 1]{Sathesis} does indeed apply.

Now to prove (b), note that $F:\cup_i X_i \to X$ is extendible, $f^{\tau_i-k}:f^k(X_i) \to X$
has bounded distortion for each $0 \le k < \tau_i$.
Consequently, also $f^k:X_i \to f^k(X_i)$ has bounded distortion.
Suppose that $|\phi(x)-\phi(y)|\le C_\phi|x-y|^\alpha$.  Since $\Phi(x) = \sum_{k = 0}^{\tau_i-1} \phi \circ f^k(x)$
for $x \in X_i$, we get for $x,y \in X_i$.
\begin{eqnarray*}
|\Phi(x) - \Phi(y)| &\le&
\sum_{k=0}^{\tau_i-1} |\phi \circ f^k(x) -  \phi \circ f^k(y)| \\
&\le&
\sum_{k=0}^{\tau_i-1} C_\phi |f^k(x) - f^k(y)|^\alpha \\
&\le&
\sum_{k=0}^{\tau_i-1} C_\phi K |f^k(X_i)|^\alpha \cdot \left(\frac{|x - y|}{|X_i|}\right)^\alpha
\end{eqnarray*}
where $K$ is the relevant Koebe constant for $F$.
Thus the condition in (b) implies that the variation $V_1(\Phi)$ is bounded.  Because $F$ is uniformly expanding, the diameter of $n$-cylinders of $F$
decreases
exponentially fast, so if $x$ and $y \in X_i$ belong to the same
$n$-cylinder,  the above estimate is exponentially small in $n$, and
summability of variations follows.
\end{proof}

The following lemma gives conditions on $f$, under which condition (b) can be used for H\"older potentials\footnote{In Lemma~\ref{lem:Holder} we take inducing schemes on a union of intervals.  As in Section~\ref{subsec:induce}, transitivity implies that this result passes to any single sufficiently small interval.}.
We say that $c \in \Crit$ has critical order $\ell_c$ if
there is a constant $C \ge 1$ \st $\frac1C |x-c|^{\ell_c} \le
|f(x)-f(c)| \le C |x-c|^{\ell_c}$ for all $x$;
$f$ is {\em non-flat} if $\ell_c < \infty$ for all $c \in \Crit$.

\begin{lemma}
Assume that $f$ is a $C^3$ multimodal map with non-flat critical points,
and let $\ell_{\max} := \max\{ \ell_c : c \in \Crit \}$.
There exists $K = K(\#\Crit, \ell_{\max})$ such that if
\[
\liminf_n |Df^n(f(c))| \ge K \quad \mbox{ for all } c \in \Crit,
\]
then formula \eqref{eq:sumalpha} holds for every $\alpha > 0$
and every inducing scheme obtained as in Proposition~\ref{prop:induce}
on a sufficiently small neighbourhood of $\Crit$. \label{lem:Holder}
\end{lemma}

\begin{proof}
We will use several results of \cite{BRSS}. First, Theorem 1 of that paper says that for any $r > 1$, we can find $\eps_0 > 0$ and
$K = K(\#\Crit, \ell_{\max}, r)$ such that if
$\liminf_n |Df^n(f(c))| \ge K \mbox{ for all } c \in \Crit$, then
the following backward contraction property holds:
Given $\eps \in (0, \eps_0)$ and
$U_\eps := \cup_{c \in \smCrit} B(f(c); \eps)$ and $s \in \N$, if $W$ is a component
of $f^{-s}(U_\eps)$ with $d(W, f(\Crit)) < \eps/r$, then $|W| < \eps/r$.

Furthermore, see \cite[Proposition 3]{BRSS},
we can find a {\em nice} set $V := \cup_{c \in \smCrit} V_c \subset f^{-1}(U_{\eps/r})$,
where each $V_c$ is an interval neighbourhood of $c \in \Crit$ and
{\em nice } means that $f^n(\partial V) \cap V = \emptyset$
for all $n \in \N$.
It follows that if $W$ is a component of $f^{-s}(V)$ contained in $V$,
then $|W| \le r^{-1/\ell_{\max}} \max_{c \in \smCrit} |V_c|$.

Proceeding by induction, and assuming that $r > 2$ is sufficiently large
to control distortion effects (cf. \cite[Lemma 3]{BRSS}),
we can draw the following conclusion.
Let $V_0$ be a component of $f^{-n}(V)$, $V_i := f^i(V_0)$ and let
$0 = t_0 < t_1 < \dots < t_k = n$ be the successive times that $V_t \subset V$.
Then $|V_{t_j}| \le 2^{j-k} \max_{c \in \smCrit} |V_c|$.

Additionally, Ma\~n\'e's Theorem implies that there are $\lambda > 1$ and $C > 0$ (depending on $V$ and $f$ only) such that
$|V_i| \le C \lambda^{-(t_j-i)} |V_{t_j}|$  for $t_{j-1} < i < t_j$.
Therefore
\[
\sum_{i=0}^n |V_i|^{\alpha}
\le \sum_{j=0}^k \sum_{m \ge 0} C^{\alpha} \lambda^{-m \alpha} |V_{t_j}|^{\alpha}
\le \frac{C^{\alpha} }{ 1- \lambda^{-\alpha}}
\sum_{j=0}^k 2^{-j \alpha} \max_{c \in \smCrit} |V_c|^{\alpha}.
\]
This implies the lemma.
\end{proof}

\section{Tail Estimates for Inducing Schemes}
\label{sec:tail}
In the following lemma, we let $\hat X \subset \hat I_{\mbox{\tiny trans}}$
be a cylinder in $\pi^{-1}(\P_N) \vee \D$
compactly contained in its domain.
This cylinder set corresponds to an
$N$-path $q$: $D \to \dots \to D_N$ in $\hat I$.
The first return map to $\hat X$ is the induced system that we will use.

The growth rate of paths in the Hofbauer tower is given by the
topological entropy. Clearly, if we remove $\hat X$ from the tower, then
this rate will decrease: we will denote it by $h^*_{top}(f)$.
If $\hat X$ is very small, then $h^*_{top}$ is close to
$\htop(f)$, so \eqref{eq:bdd range} implies
that $\sup \phi - \inf \phi < h^*_{top}$
for  $\hat X$ sufficiently small.
Note that we can in fact take $\hat X$ to be the type of set, a union of domains in $\hat I$, considered in \cite{BrCMP}.  We will use this type of domain in Section~\ref{sec:ana}.

\begin{proposition}\label{prop:exp_tail}
Suppose that $V_n(\phi) \to 0$ and
let $\hat \psi = \hat\phi - P_G(\hat \phi,\hat X)$. If $\hat X \in \hat\P_N$ is so small that
\[
\sup \phi - \inf \phi < h^*_{top},
\]
then there exist $C, \gamma > 0$ \st
$Z^*_n(\hat \psi, \hat X) < C \, e^{-\gamma n}$.
\end{proposition}

\begin{proof}
We will approximate $Z_n^*(\hat \phi, \hat X)$ by adding
the weights $e^{\hat \phi_{n-1}(\hat x)}$ of all $n-1$-paths from
$\hat f(\hat X)$ to $\hat X$ in the Hofbauer tower
with outgoing arrows from $\hat X$ removed.
By removing these arrows we ensure that these paths will not visit
$\hat X$ before step $n$, so we indeed approximate $Z_n^*(\hat \phi, \hat X)$
and not $Z_n(\hat \phi, \hat X)$.
In considering $n-1$-paths,
we only miss the initial
contribution $e^{\hat\phi|_{\hat X}}$ in the weight
$e^{\hat \phi_n(\hat x)}$ for $\hat x = \hat f^n(\hat x) \in \hat X$,
so it will not effect the exponential growth rate
$P^*_G(\hat \phi, \hat X)$ of $Z^*_n(\hat \phi, \hat X)$.
Since $Z^*_n(\hat \psi, \hat X) = e^{-nP_G(\hat \phi, \hat X)} Z^*_n(\hat \phi, \hat X)$, the proposition follows if we can show the strict inequality $P^*_G(\hat \phi, \hat X) <  P_G(\hat \phi, \hat X)$.

{\bf Remark:} It is this strict inequality that is responsible for the
discriminant $\dm_F[\phi]$ in Section~\ref{sec:ana}
being strictly positive.

{\bf The rome technique:}
We will  approximate the Hofbauer tower by finite Markov
graphs, and use the following general idea of {\em romes} in transition
graphs from Block et al.
\cite{BGMY} to
estimate $Z_n^*(\hat \phi, \hat X)$.
Let $\G$ be a finite graph where every edge $i \to j$ has a weight
$w_{i,j}$, and let $W = (w_{i,j})$ be the corresponding (weighted)
transition matrix. More precisely, $w_{i,j}$ is the total weight of all
edges $i \to j$, and if there is no edge $i \to j$, then
$w_{i,j} = 0$.

A subgraph
$\rome$ of $\G$ is called a {\em rome}, if there are no loops
in $\G \setminus {\rome}$. A \emph{simple path}
$p$ of length $l(p)$ is given by
$i = i_0 \to i_1 \to \dots \to i_{l(p)} = j$, where $i,j \in \rome$,
but the intermediate vertices belong to
$\G \setminus \rome$. Let
$w(p) = \prod_{k = 1}^{l(p)} w_{i_{k-1}, i_k}$ be the weight of $p$.
The {\em rome matrix} $A_{\smromea}(x) = (a_{i,j}(x))$, where $i,j$ run
over the vertices of $\rome$, is given by
$$
a_{i,j}(x) = \sum_p w(p)x^{1-l(p)},
$$
where the sum runs over all simple paths $p$ as above.
(Note that with the convention that $x^0 = 1$ for $x = 0$,
$A_{\smromea}(0)$ reduces to the weighted transition matrix of
the rome $\rome$.)
The result from \cite{BGMY} is that
the characteristic polynomial of $W$ is equal to
\begin{equation}\label{eq:rome}
 \det( W - xI_W) = (-x)^{\# \G - \# {\rome} } \det( A_{\smromea}(x) - xI_{\smromea} ),
\end{equation}
where $I_W$ and $I_{\smromea}$ are the identity matrices of the appropriate
dimensions.

In our proof, we will use $k$-cylinder sets as vertices in
the graph $\G$, and we will take $w(p) = e^{\hat\phi_{l(p)}(x)}$
for some $x$ belonging to the interval in $\hat I$
that is represented by the path $p$.

{\bf Choice of the rome:} Fix a large integer $k$.
The partition $\hat \P_k$ is clearly a Markov partition for the
Hofbauer tower, and its dynamics can be expressed by a countable
graph $(\hat \P_k, \to )$, where $\hat P \to \hat Q$ for
$\hat P, \hat Q \in \hat \P_k$
only if $\hat f(\hat P) \supset \hat Q$.
Choose $R \gg k$ (to be determined later).
Given a domain $D$ of level $R$, from all the $R$-paths
starting at $D$, at most two (namely those corresponding the
the outermost $R$-cylinders in $D$) avoid $\hat I_R$.
Any other $R$-path from $D$ has a shortest subpath $D \to \dots \to D'$
where both $D$ and $D' \in \hat I_R$.
Let us call the union of all points in $\hat I$ that belong to one of
such subpaths the {\em wig} of $\hat I_R$.

The vertices of the rome $\rome$ are those
cylinder sets $\hat P \in \hat \P_k$, $\hat P \not\subset \hat X$,
that are either contained in domains $D \in \D$
of level $< R$, or that belong to the wig.
We retain all arrows between two vertices in $\rome$.
Let $A_{\smrome}$ be the weighted transition matrix of $\rome$.
For each arrow $\hat P \to \hat Q$, choose $\hat x \in \hat P$ such that
$\hat f(\hat x) \in \hat Q$, and set $w_{\hat P, \hat Q} = e^{\hat
\phi(\hat x)}$. Let $\rho_{\smrome}$
be the leading eigenvalue of the weighted transition matrix.
The pressure $P^*_G(\hat \phi)$ is approximated (with error of order
$V_k(\hat\phi)$) by $\log \rho_{\smrome}$.

The graph $({\rome},\to)$ is a finite
subgraph of the full infinite Markov graph $(\hat \P_k, \to )$.
We will construct two other finite graphs $(\G_0, \to)$ and $(\G_1,\to)$
both having $\rome$ as a rome, and minorising respectively majorising
$(\hat\P_k, \to)$ in the following sense:
For each path in $(\G_0, \to)$,
including those passing through $\hat X$,
we can assign a path in $(\hat \P_k, \to)$ of comparable weight,
and this assignment can be done injectively.
Conversely,
for each path in $(\hat \P_k, \to)$, except those passing through $\hat X$,
we can assign a path in $(\G_1, \to)$ of comparable weight,
and this assignment can be done injectively.

As $\rome$ is a rome to both $\G_0$ and $\G_1$, we can use the rome
technique to compare the spectral radii $\rho_0$ and $\rho_1$
of their respective weighted transition matrices
$W_0$ and $W_1$.
By the above minoration/majoration property,
we can separate $e^{P^*_G(\phi)}$ from
$e^{P_G(\phi)}$ by $\rho_0$ and $\rho_1$,
up to a distortion error. By refining the partition of the Hofbauer
tower into $k$-cylinders, \ie taking $k$ large,
whilst maintaining the majoration/minoration property,
we can reduce the distortion error (relative to the iterate),
and also show that $\rho_0 < \rho_1$.
This will prove the strict inequality
$P^*_G(\hat \phi) < P_G(\hat \phi)$.

{\bf The graph \boldmath $\G_0$\unboldmath:}
First, to construct $\G_0$, we add the arrows $\hat P \to \hat Q$
for each $\hat P \in \hat \P_k \cap \hat X$ and $\hat Q \in \hat \P_k$
\st $\hat f(\hat P) \supset \hat Q$.
The weight of this arrow is $e^{\hat \phi(\hat x)}$
for some chosen $\hat x \in \hat P$.
Let $W_0$ be the weighted transition matrix of $\G_0$. It
follows that its spectral radius is a lower bound for $e^{P_G(\hat
\phi)}$, up to an error of order $e^{V_k(\hat \phi)}$. Furthermore,
the number of $n$-paths in $\rome$ is at least $e^{n(\htop(f)-\eps_R)}$,
where $\eps_R \to 0$ as $R \to \infty$, cf. \cite{Hpwise}.
Since each arrow has weight at least $e^{\inf \hat \phi}$,
we obtain
\begin{equation}\label{eq:rho0a}
e^{\htop(f) + \inf \hat \phi - \eps_R } \le \rho_0 := \rho(W_0) \le
e^{P_G(\hat \phi) + V_k(\hat \phi) }.
\end{equation}
Let $L$ be such that $f^L(\pi(\hat X)) \supset I$.

Let $v = (v_{\hat P})_{\hat P \in \P_k}$ be the positive left unit
eigenvector corresponding to the
leading eigenvalue $\rho_{\smrome}$ of  $A_{\smrome}$.
Recall that for each $R_0 \in \N$ and $D \in \hat I$ there are at most
two $R_0$-paths from $D$ leading to domains of level $> R_0$.
Each such path corresponds to a subintervals of $D$ adjacent to $\partial D$,
and although this subinterval may consist of many adjacent
cylinder sets of $\hat \P_k$,
$\hat f^R$ maps them monotonically onto adjacent cylinder sets
of $\hat P_{k-R_0}$.
Therefore
\begin{eqnarray*}
\sum_{D \in \D \atop\slev(D) > R_0} \sum_{\hat Q \in \P_k \cap D} v_{\hat Q}
&=& \frac{1}{\rho_{\smrome}^{R_0}}
\sum_{\slev(\hat Q) > R_0} \sum_{\hat P \in \P_k}
v_{\hat P} (A_\smrome^{R_0})_{\hat P, \hat Q} \\
&\le& 2 e^{\sup \hat\phi_{R_0}} \rho_{\smrome}^{-R_0} \sum_{\hat P \in \P_k} v_{\hat P} \\
&=& 2 e^{\sup \hat \phi_{R_0} - R_0 P_G(\hat \phi)}
\leq 2 e^{R_0 (\sup \hat \phi - \inf \hat \phi - \htop(f)) },
\end{eqnarray*}
independently of $k$.
Since $\sup \hat \phi - \inf \hat \phi - \htop(f) < 0$,
we can take $R_0$ so large, independently of $k$, that for every $x \in I$,
\begin{equation}\label{eq:R0}
\sum_{\hat Q \in \hat \P_k,\  \pi(\hat Q) \owns x \atop \slev(\hat Q) > R_0} v_{\hat Q}
< \frac12 \min\{  v_{\hat Q} : \hat Q \in \P_k \cap \hat f^L(\hat X), \pi(\hat Q) \owns x\}.
\end{equation}

The idea is now to offset all contributions of
$n$-paths starting from level $> R_0$ to $Z^*_n(\hat \phi, \hat X)$
by the contribution of $n$-paths starting in $\hat f^L(\hat X)$ to
 $Z_n(\hat \phi, \hat X)$.
Let $N \ge L$ be such that there is an $N$-path from $\hat X$ to
every $\hat Q$ of level $\le R_0$.
Then
\begin{eqnarray}\label{eq:v}
\left( v W_0^N \right)_{\hat Q} &\ge &
\left( v (A_{\rome}^N + e^{-\inf \hat \phi_N} \Delta) \right)_{\hat Q} \nonumber \\
&\ge &
\left\{  \begin{array}{ll}
(\rho_{\smrome}^N +  e^{-\inf \hat \phi_N} \kappa) v_{\hat Q} & \text{ if } \lev(\hat Q) \le R_0, \\
\rho_{\smrome}^N v_{\hat Q} & \text{ if } \lev(\hat Q) > R_0,
\end{array} \right.
\end{eqnarray}
where
$\kappa := \min\{ v_{\hat Q} : \hat Q \in \hat \P_k \cap \hat f^L(\hat X)\} /
\max v_{\hat Q}$, and $\Delta$ a nonnegative square matrix with some
$1$s in the rows corresponding to $\hat P \in \hat \P_k \cap \hat f^L(\hat X)$
in such a way that the column corresponding to each $\hat Q$
with $\lev(\hat Q) \le R_0$ has at least one $1$.
The fact that $\kappa > 0$ uniformly in the order of cylinder sets $k$
rests on the following claim, which is proved later on:
\begin{equation}\label{eq:claim}
\min_{\hat Q \in \hat \P_k \cap \hat f^L(\hat X)}
v_{\hat Q} / \max_{\hat Q\in \hat \P_k} v_{\hat Q} > 0
\qquad \mbox{ uniformly in } R \mbox{ and }k.
\end{equation}
By the choice of $L$, $R_0$ (see \eqref{eq:R0}) and $N$,
\begin{equation}\label{eq:R0a}
\sum_{\hat Q \in \hat \P_k, \slev(\hat Q) > R_0} e^{\hat \phi_N(\hat Q)}
v_{\hat Q} (A_{\smrome}^N)_{\hat Q, \hat P}
\le \frac12
\sum_{\hat Q \in \hat \P_k \cap \hat f^L(\hat X)}
v_{\hat Q} (W_0^N)_{\hat Q, \hat P}
\end{equation}
for each $\hat P$ with $\lev(\hat P) \le R_0$.
When we apply $W_0^N$ to \eqref{eq:v} once more, the components
$v_{\hat Q}$ with $\lev(\hat Q) \le R_0$ have increased by a factor
$\rho_{\smrome}^N + e^{-\inf \hat \phi_N} \kappa$, whereas by \eqref{eq:R0a},
the components $v_{\hat Q}$
with $\lev(\hat Q) > R_0$ combined amount to at most half the weight of
the components $v_{\hat Q}$ with $\hat Q \in \hat \P_k \cap \hat f^L(\hat X)$.
Therefore, we can generalise  \eqref{eq:v} inductively to
\begin{eqnarray}\label{eq:vm}
\left( v W_0^{Nm} \right)_{\hat Q} &\ge &
\left( v (A_{\rome}^N + e^{-\inf \hat \phi_N} \Delta)^m \right)_{\hat Q}
 \nonumber \\
&\ge &
\left\{  \begin{array}{ll}
(\rho_{\smrome}^N + \frac12 e^{-\inf \hat \phi_N} \kappa)^m v_{\hat Q} & \text{ if } \lev(\hat Q) \le R_0, \\
\rho_{\smrome}^{Nm} v_{\hat Q} & \text{ if } \lev(\hat Q) > R_0,
\end{array} \right.
\end{eqnarray}
for all $m \ge 1$. It follows that
\begin{eqnarray*}
\sum_{\slev(\hat Q) \le R_0}
\frac{1}{\rho_0^{Nm}} (\rho_{\smrome}^N+ \frac12 e^{-\inf \hat \phi_N}\kappa)^m v_{\hat Q}
&\le& \sum_{\slev(\hat Q) \le R_0}
\frac{1}{\rho_0^{mN}} \left( v W_0^{mN} \right)_{\hat Q} \\
&\to& \alpha \!\!\!\!\!
\sum_{\slev(\hat Q) \le R_0} w_{\hat Q}
\end{eqnarray*}
for some $\alpha < \infty$ and $w$ the left unit eigenvector
corresponding to the leading eigenvalue $\rho_0$ of $W_0$.
This implies that
\begin{equation}\label{eq:rho0}
\rho_0 \ge (\rho_{\smrome}^N+\frac12  e^{-\inf \hat \phi_N} \kappa)^{1/N}
\quad \text{ whence } \quad  \rho_0 > \rho_{\smrome}+\kappa'
\end{equation}
for some $\kappa' = \kappa'(\kappa,N,\hat \phi, R_0) > 0$, uniformly
in $R \ge R_0$ and $k \in \N$.

{\bf The graph \boldmath $\G_1$\unboldmath:}
For each $\hat P \in \hat \P_k \cap D$
where $D$ has level $R$, consider all $R$-paths $p:\hat P \to \dots \to \hat
Q$ that avoid $\hat I_R$; these are not included in $(\rome,
\to)$. From each $D$ of level $R$, there at
most $2$ such $R$-paths avoiding $\hat I_R$,  corresponding to $R$-cylinders
in $D$. These two $R$-cylinders are contained in two $k$-cylinders in $D$.
For each such $k$-cylinder $\hat P$ (\ie vertex in $(\hat \P_k , \to)$),
and each $Q \in \P_k \cap f^R(\pi(\hat P))$,
choose $\hat Q \in \hat P_k \cap \hat I_R$ and
attach an artificial $R$-path with $R-1$ new vertices and a terminal vertex
$\hat Q$.
Assign weight $w(p) = e^{R\sup\hat\phi}$ to this path.
Therefore, if $f$ is $d$-modal,
the number of vertices added to $\hat I_R$ is therefore no larger that
$2d(R-1)$.
Call the resulting graph $\G_1$ and $W_1$ its weighted
transition matrix.

Any $n$-path in the Hofbauer tower that leaves $\hat I_R$ for at
least $R$ iterates can be mimicked by an $n$-path following one of
the additional $R$-paths in $\G_1$. But $n$-orbits visiting $\hat X$
are still left out. It follows that this time, the leading
eigenvalue estimate exceeds the exponential growth rate
of the contributions of all $n$-periodic orbits in the Hofbauer
tower that avoid $\hat X$.
Since the error of order $e^{V_k(\hat \phi)}$
still needs to be taken into account, we get
\begin{equation}\label{eq:rho1a}
\rho_1 := \rho(W_1) \ge e^{P_G^*(\hat\phi) - V_k(\hat \phi)}.
\end{equation}
On the other hand, we can use \eqref{eq:rome} to deduce that
\begin{equation}\label{eq:A1}
\det( W_1 - x I_{W_1}) = (-x)^{\# \G_1 - \#{\rome} }
\det( A_1(x) - xI_{\smrome} ),
\end{equation}
where the rome matrix $A_1(x)$ equals $A_{\smrome}$,
except for new entries
$w_{\hat P, \hat Q} \le e^{R \sup \hat \phi} x^{1-R}$
for the $R$-path added to the rome.
These paths correspond to $R$-cylinders, at most $2$ for each
of the $d$ domains of level $R$, and since $R \ge k$,
there are at most $2d$ paths with initial
vertices $\hat P \in \P_k$, each with
at most $\#\P_k$ terminal vertices $\hat Q$.
In other words, $A_1(x) \le
A_{\smrome} + x^{1-R} e^{R \sup \hat \phi} \Delta_1$, where $\Delta_1$ is a
square matrix with at most $2d$ non-zero rows
(corresponding to initial vertices $\hat P$) and zeros otherwise.
Formula \eqref{eq:A1} shows that $\rho_1$ is also the
leading eigenvalue of $A_1(\rho_1)$.

Although matrices
$A_{\smrome}$ and $\rho_1^{1-R} e^{R \sup \hat \phi} \Delta_0$ depend both on
$R$ and $k$, at the moment we will only need $k$ so large that
\begin{equation}\label{eq:boundk}
V_k \le \alpha := \frac12\left(\htop^*(f) - (\sup \hat \phi - \inf \hat \phi)\right)
\end{equation}
and hence suppress the dependence on $k$ until it is needed again.

We first give some estimates necessary to apply
Lemma~\ref{lem:matrices} below with $U_R = A_{\smrome}$ and
$V_R  = \rho_1^{1-R} e^{R \sup \hat \phi} \Delta_1$.
The `left' matrix norm (which is the maximal row-sum) of $\Delta_1$ is
$\| \Delta_1\| := \sup_{ \| v \|_1 = 1 } \| v \Delta_1\|_1 = \# \P_k$,
and therefore (using also \eqref{eq:rho1a}) we obtain
\begin{eqnarray*}
\| \rho^{1-R} e^{\sup \hat \phi} \Delta_1\|
&\le& \# \P_k \rho_1^{1-R} e^{R \sup \hat \phi}\\
&\le&  \# \P_k e^{R (\sup \hat \phi - P_G^*(\hat \phi) + V_k(\hat \phi))}\\
&\le& \# \P_k e^{R(\sup \hat \phi - \inf \hat \phi - \htop^*(f) + V_k(\hat \phi))}
\le  \# \P_k e^{- \alpha R}
\end{eqnarray*}
for $\alpha > V_k(\hat \phi)$ as in \eqref{eq:boundk}.
The entries of $(A_{\smrome}^m)_{\hat P, \hat Q}$ indicate the sum of
the weights of all $m$-paths from $\hat P$ to $\hat Q$.
For each $\hat Q \in \hat P_k$, the sum
\[
\sum_{\pi(\hat Q) \subset Q}\ \sum_{\mbox{\tiny paths } \hat P \to \hat Q}
e^{\sup \hat\phi_m|_{\hat P}} \le \#\{ \text{components of } \pi(\hat P) \cap
f^{-m}(Q) \}\ e^{\sup \phi_m|_{\pi(\hat P)}},
\]
which has exponential growth-rate $\rho_{\smrome}$.
Therefore the left matrix norm
$\| A_{\smrome}^m \| \leq \rho_{\smrome}^m e^{\eta m}$
for some $\eta = \eta(R,k)$ with $\lim_{R \to \infty} \eta(R,k) = 0$
for each fixed $k$.

If $v'$ is the positive left eigenvector of $A_1(\rho_1)$,
corresponding to $\rho_1$ and  normalised so that $\| v' \|_1
:= \sum_i |v'_i| = 1$, then
\begin{eqnarray} \label{eq:rho1}
\rho_1 &=& \| v' \rho_1 \|_1 = \| v' (A_1(\rho_1)^m  \|_1^{1/m}  \nonumber  \\[2mm]
&=& \| (A_{\smrome} + \rho_1^{1-R} e^{R \sup \hat \phi} \Delta_1)^m \|^{1/m} \nonumber  \\[2mm]
&\le&  \rho_{\smrome} \left(1+ \| A_{\smrome} \| e^{m \tilde \eta(R,k) } \right)^{1/m} \to
 \rho_{\smrome} e^{\tilde \eta(R,k) } \qquad \text{ as } m \to \infty,
\end{eqnarray}
where $\tilde \eta(R,k)$ comes from Lemma~\ref{lem:matrices}.

Using \eqref{eq:rho1} and \eqref{eq:rho0}
we obtain
\[
\rho_1 \le \rho_{\smrome} e^{\tilde \eta(R,k)}
\le  e^{\tilde \eta(R,k)} (\rho_0 - \kappa')
\]
By claim \eqref{eq:claim},
$\kappa' > 0$ uniformly in $R$ and $k$, and by Lemma~\ref{lem:matrices},
we can choose $R$ large (and hence $\tilde \eta(R,k)$ small)
to derive that $\rho_1 < \rho_0$.
It follows by \eqref{eq:rho0a} and \eqref{eq:rho1a}  that
\[
e^{ P_G(\hat \phi) + V_k(\hat \phi) } \ge \rho_0 > \rho_1
\ge e^{P^*_G(\hat\phi) - V_k(\hat \phi)}.
\]
so taking the limit $k \to \infty$, we get
$P_G(\hat \phi) > P^*_G(\hat\phi)$
as required.

{\bf Proof of Claim \eqref{eq:claim}:}
We start with the uniformity in $R$, \ie
the level at which the Hofbauer tower is cut off.
Recall that we assumed that $\hat X$ is so small that
$\sup \hat \phi - \inf \hat \phi < h^*_{top}$.
The leading eigenvalue $\rho_{\smrome}$ of $A_{\smrome}$ satisfies
$\rho_{\smrome} \ge e^{P^*_G(\hat \phi) - V_k(\hat\phi) - \eps_R}$
(see \eqref{eq:rho0a}), because $\rome$ is a subgraph
of the Hofbauer tower with $\hat X$ removed.
For any $r$ and any domain $D \in \hat I$,
there are at most two $r$-paths ending outside $\hat I_r$.
Therefore if $\hat P, \hat Q \in \hat \P_k$ where $\hat Q$ is
contained in a domain $D$ of level $\ge r$,
the $\hat P, \hat Q$-entry of $A_{\smrome}^r$ is at
most $2 e^{r \sup \hat \phi}$.
Thus we find for the left
eigenvector $v$
\[
\rho_{\smrome}^r \sum_{\hat Q \in \hat \P_k \cap D} v_{\hat Q} =
 \sum_{\hat Q \in \hat \P_k \cap D} \left( v A_{\smrome}^r \right)_{\hat Q}
\le 2 e^{r \sup \hat \phi} \sum_{\hat P \in \hat \P_k} v_{\hat P}
\le 2 e^{r \sup \hat \phi}.
\]
It follows that
\[
 \sum_{\hat Q \in \hat \P_k \cap D} v_{\hat Q} \le
2 e^{r (\sup \hat\phi - P^*_G(\hat \phi) + V_k(\hat \phi) + \eps_R ))}
\le
2 e^{r (\sup \hat\phi - \inf \hat \phi - \htop^*(f) +
V_k(\hat \phi) + \eps_R)}
\]
is exponentially small in $r$.
There are at most $2d$ domains $D$ of level $r$, which implies that
$\sum_{\slev(\hat Q) > r} v_{\hat P}$ is exponentially
small in $r$, and this is independent of $R \ge r$, and of how
(or whether) the Hofbauer tower is truncated.

Next take $r_0$ so large that
$\sum_{\slev(\hat Q) > r_0} v_{\hat P} < \frac12$ irrespective of the
way the Hofbauer tower is cut, and such that $\hat X$ belongs to
a transitive subgraph of $\hat I_{r_0}$.
Therefore there is $r'_0$ such that for every
domain $D$ of $\lev(D) \le r_0$ and every $\hat Q \in \hat \P_k \cap D$,
there is an $r'_0$-path from $\hat Q$ to $\hat X$.
Hence the $\hat Q, \hat P$ entry in
$A_{\smrome}^{r'_0}$ is at least $e^{r'_0 \inf \hat \phi}$
for every $\hat P \in \hat \P_k \cap \hat X$.
Since $v = v ( \rho_{\smrome}^{-1} A_{\smrome})^{r'_0}$, we find
\[
\sum_{\hat P \in \hat \P_k \cap \hat X} v_{\hat P}
\ge \rho_{\smrome}^{-r'_0}
e^{r'_0 \inf \hat \phi}\sum_{\hat Q \in \hat \P_k \cap \hat I_{r_0}} v_{\hat Q}
\ge  \frac12 \ \rho_{\smrome}^{-r'_0}
e^{r'_0 \inf \hat \phi}
\]
independently of $R \ge r_0$.

Now we continue with the uniformity in $k$.
This is achieved by analysing the effect of
splitting of vertices of the transition graph into new vertices, representing
cylinders of higher order. We do this one vertex at the time.

Let $W$ be a weighted transition matrix of a graph $\mathcal G$.
Given a vertex $g \in \mathcal G$, we can represent the $2$-paths
from $g$ by splitting $g$ as follows (for simplicity, we assume that
the first row/column in $W$ represents arrows from/to $g$):
\begin{itemize}
\item If $g \to_{w_{1,b_1}} b_1, g \to_{w_{1,b_2}} b_2, \dots,
g \to_{w_{1,b_m}} b_m$ are the outgoing arrows,
replace $g$ by $m$ vertices $g_1, \dots, g_m$
with outgoing arrows
$g_1 \to_{w_{1,b_1}} b_1, g_2 \to_{w_{1,b_2}} b_2,
\dots, g_m \to_{w_{1,b_m}} b_m$ respectively, where $w_{1,b_j}$
represents the weight of the arrow.
\item Replace all incoming arrows $c  \to_{w_{c,1}} g$
by $m$ arrows
$c  \to_{w_{c,1}} g_1, c  \to_{w_{c,1}} g_2, \dots,
c  \to_{w_{c,1}} g_m$, all with the same weight.
\item If $g \to g$ was an arrow in the old graph, this means
that $g_1$ will now have $m$ outgoing arrows:
$g_1  \to_{w_{1,1}} g_1, g_1 \to_{w_{1,1}} g_2, \dots,
g_1 \to_{w_{1,1}} g_m$, all with the same weight.
\end{itemize}

\begin{lemma}\label{lem_vertex_split}
If $W$ has leading eigenvalue $\rho$ with left eigenvector
$v = (v_1, \dots, v_n)$, then the weighted transition matrix $\tilde W$
obtained from the above procedure has again $\rho$ as leading eigenvalue,
and the corresponding left eigenvector is
$\tilde v = (\underbrace{v_1, \dots, v_1}_{m\ \mbox{\tiny times}},
v_2, \dots v_n)$.
\end{lemma}

\begin{proof}
Write $W = (w_{i,j})$ and assume that $w_{1,1} \neq 0$, and the other non-zero
entries in the first row are $w_{1,b_2}, \dots, w_{1,b_m}$.
The multiplication $\tilde v \tilde W$ for the new matrix and eigenvector
becomes
\begin{align*}
& \qquad \quad {}\overbrace{\qquad}^{m\ \mbox{\tiny times}}\\
(\underbrace{v_1, \dots, v_1}_{m\ \mbox{\tiny times}},
v_2, \dots v_n) &
\left( \begin{array}{ccc|ccccccc}
w_{1,1} & \dots & w_{1,1} & 0 & \dots & & & & \dots & 0 \\
0 & \dots & 0 & \dots & 0 & w_{1,b_2} & 0 & & &\\
\vdots & & \vdots & & & & & & &\\
0 & \dots & 0 & \dots &  & &  0 & w_{1,b_m} & 0 & \dots \\
\hline
w_{2,1} & \dots & w_{2,1} & w_{2,2} & \dots & & & & \dots & w_{2,n} \\
w_{3,1} & \dots & w_{3,1} & w_{3,2} & \ddots  & & & & & \vdots  \\
\vdots & & \vdots  & & & & & & &   \\
\vdots & & \vdots & \vdots & & & & & \ddots & \vdots  \\
w_{n,1} & \dots & w_{n,1} & w_{n,2}  & \dots & & & & \dots & w_{n,n}
\end{array} \right)
\end{align*}
A direct computation shows that this equals $\rho \tilde v$.
Since $\tilde v$ is positive, it has to belong to the leading eigenvalue, so
$\rho$ is the leading eigenvalue of $\tilde W$ as well.
The proof when $w_{1,1} = 0$ is similar.
\end{proof}

The effect of going from $\hat \P_k$ to $\hat \P_{k'}$
for $k' > k$ is that by repeatedly
applying  Lemma~\ref{lem_vertex_split},
the entries $v_{\hat P}$ for $\hat P \in \hat \P_k$
have to be replaced by $\# (\hat P \cap \hat \P_{k'})$
copies of themselves which, when normalised, leads
to the new unit left eigenvector $\tilde v$.
If $\pi(\hat P) \subset \pi(\hat Q)$, then
the number of $k'$-cylinders in $\hat P$ is less than
the number of $k'$-cylinders in $\hat Q$.
Since $I$ contains a finite number of $k$-cylinders,
there is $C = C(k)$ such that
$\#(\hat P \cap \P_{k'}) \le C \#(\hat Q \cap \P_{k'})$
for all $\hat P, \hat Q \in \P_k$ and $k' > k$.
When passing from $\hat \P_k$ to $\hat \P_{k'}$, we also need to
to adjust the weight $e^{\hat \phi(x)}$ for $x \in \hat P \in \hat \P_k$
slightly,
but this adjustment is exponentially small since $V_{k'}(\hat \phi) \to 0$.
It follows that
$\min_{\hat P \in \hat \P_{k'} \cap \hat f^L(\hat X)} v_{\hat P} / \max v_{\hat P}$ is uniformly bounded away from $0$, uniformly in $k'$.
\end{proof}

We finish this section with the technical result used \eqref{eq:rho1}.

\begin{lemma}\label{lem:matrices}
Let $\{ U_n\}_{n \in \N}$, $\{ V_n \}_{n \in \N}$ be positive square matrices
such that $\rho_n \ge 1$ is the leading eigenvalue of $U_n$.
Assume that there exist $M < \infty$, $\tau\in (0,1)$ and
a sequence $\{\eta_k \}_{k \in \N}$
with $\eta_k \downarrow 0$ as $k \to \infty$ such that for all $n$
\[
\| U_n \| \le M, \quad  \| U_n^k \| \le \rho_n^k e^{k \eta_k}
\quad \text{ and } \quad \| V_n\|\le M\tau^n,
\]
Then there exists a different sequence $\{\tilde \eta_n\}_{n \in \N}$
with $\tilde \eta_n \to 0$ as $n \to \infty$
such that
\[
\|(U_n+V_n)^j\| \le (1 +  e^{ j \tilde \eta_n} ) \rho_n^j.
\]
In particular, the leading eigenvalue of  $\frac{1}{\rho_n}(U_n+V_n)$
tends to $1$ as $n\to \infty$.
\end{lemma}

\begin{remark}
Although this lemma works for any matrix norm, we need it
for $\| U \| = \sup_{\| v \|_1 = 1} \| v U \|_1$, \ie the maximal row-sum of $U_n$.
Note that we do not assume that all $U_n$ have the same size
(although $U_n$ and $V_n$ have the same size for each $n$).
\end{remark}

\begin{proof}
Note that $U_n+V_n$ is a positive matrix
and so its leading eigenvalue is equal to the growth rate
$\lim_{j \to \infty} \frac1j \log  \|(U_n+V_n)^j\|$.
We have
\[
(U_n+V_n)^j =
\sum_{|p|+|q|=j} U_n^{p_1}V_n^{q_1}\dots
U_n^{p_t}V_n^{q_t},
\]
where $p=(p_1,\dots,p_t)$, $q=(q_1,\dots,q_t)$
and  $|p|=\sum p_i$ and $|q|=\sum q_i$.
More precisely, the sum runs over all
$t \in \{ 1, \dots, \lceil j/2 \rceil \}$  and distinct vectors $p,q$ with
$p_i,q_i > 0$ (except that possibly $p_1 = 0$ or $q_t=0$).
Let us split the above sum into two parts.\\
(i) If $|q| > \eps j$, then each of the above terms can be estimated in norm
by
\[
\|U_n\|^{|p|} \| V_n\|^{|q|} \le M^j (\tau^n)^{\eps j}
= (M\tau^{\eps n})^{j}.
\]
Since there are at most $2^j$ such terms, this gives
\begin{equation}\label{eq:parti}
\biggl\|\sum_{|p|+|q|=j \atop |q| > \eps j} U_n^{p_1}V_n^{q_1}\dots
U_n^{p_t}V_n^{q_t}\biggr\|  \le (2M\tau^{\eps n})^{j}.
\end{equation}
(ii) If $(q_1, \dots, q_t)$
satisfies $|q| \le \eps j$, then there are at most
$t-1 \le |q|$ indices $i$ with $p_i \le N$
and at least one index $i$ with $p_i > N$,
where $N < 1/(2\eps)$  is to be determined later.
The norm of each of these terms can be estimated
by $\| U_n^{p_1}\| \cdots \| U_n^{p_t}\| M^{|q|} \tau^{n|q|}$, where
the factors
\[
\| U_n^{p_i}\| \le \left\{ \begin{array}{ll}
\rho_n^{p_i} e^{\eta_N p_i} & \text{ if } p_i > N, \\
M^N & \text{ if } p_i \le N.
\end{array} \right.
\]
So the product of all these factors is
at most $\rho_n^j e^{\eta_N j} M^{\eps j N}$.
Using Stirling's formula, we can derive that there are at most
\[
\sum_{t=0}^{\lfloor \eps j \rfloor } { j \choose t }
\le \eps j {j \choose \lfloor \eps j \rfloor}
\le
\sqrt{\eps j} \left( \frac{1}{\eps} \right)^{\eps j}
\left( \frac{1}{1-\eps} \right)^{(1-\eps) j} \leq e^{\sqrt{\eps} j}
\]
possible terms of this form.
Combining all this gives an upper bound of this part of
\begin{equation}\label{eq:partii}
\biggl\|\sum_{{|p|+|q|=j} \atop {|q|\le \eps j}} U_n^{p_1}V_n^{q_1}\cdots
U_n^{p_t}V_n^{q_t} \biggr\|
\le e^{\sqrt{\eps} j} \rho_n^j e^{\eta_N j} M^{\eps j N} (M \tau^n)^{\eps j}.
\end{equation}
Adding the estimates of \eqref{eq:parti} and  \eqref{eq:partii}, we get
\[
\|(U_n+V_n)^j\| \le (2M \tau^{\eps n})^{j} +
 e^{\sqrt{\eps} j} \rho_n^j e^{\eta_N j} M^{\eps N j} (M\tau^n)^{ \eps j}.
\]
Now take $N = n^{\frac14}$ and $\eps = n^{-\frac12}$
(so indeed $N < 1/(2\eps)$) and $n$ so large that
$M \tau^n \le 2M \tau^{\sqrt n} \le 1$. Then we get
\[
\|(U_n+V_n)^j\| \le \rho_n^j \left( 1 +
e^{j ( n^{-1/4}  + \eta_{n^{1/4}} + n^{-1/4} \log M )}  \right).
\]
The lemma follows with $\tilde \eta_n = ( n^{-1/4}  + \eta_{n^{1/4}} + n^{-1/4} \log M )$.
\end{proof}

\section{Proof of Theorem~\ref{thm:exist for bdd range}}
\label{sec:exist for bdd range}

The following is \cite[Lemma 3]{BTeqnat}.
\begin{lemma}\label{lem:compat entro}
For every $\eps > 0$, there are $R \in \N$ and $\eta > 0$ such
that if $\mu \in \M_{erg}$ has entropy $h_\mu(f) \ge \eps$, then
$\mu$ is liftable to the Hofbauer tower and $\hat \mu(\hat I_R)
\ge \eta$. Furthermore, there is a set $\hat E$, depending only on
$\eps$, \st $\hat \mu(\hat E)> \eta/2$ and $\min_{D\in \D\cap\hat
I_R}d(\hat E\cap D, \bd D)>0$.
\end{lemma}

The following lemma will allow us to implement condition (c) in Proposition~\ref{prop:all_indu}.

\begin{lemma}
There exist sequences $\{\eps_n\}_n\subset \R^-$ with $\eps_n \to 0$ and $\{\mu_n\}_n\subset \M_+$ so that $h_{\mu_n}(f)+\int\psi~d\mu_n \ge \eps_n$.
Moreover, there exists a domain $\hat X$ compactly contained in some $D\in \D$  so that $\hat\mu_n(\hat X)>0$. \label{lem:choice of X}
\end{lemma}

\begin{proof}
First notice that by the definition of pressure, there must exist sequences $\{\eps_n\}_n\subset \R^-$ with $\eps_n \to 0$ and $\{\mu_n\}_n\subset \M_{erg}$ so that $h_{\mu_n}(f)+\int\psi~d\mu_n \ge \eps_n$.  By \eqref{eq:lmuhmu}, there exists $\eps>0$ so that we can choose $h_{\mu_n}(f)>\eps$ and $\{\mu_n\}_n\subset \M_+$.  Now by Lemma~\ref{lem:compat entro}, we can choose $\hat X$ compactly contained in some $D\in \D$ and a subsequence $\{n_k\}_k$ with $\hat\mu_{n_k}(\hat X)>0$ for all $k$.
\end{proof}

\begin{proof}[Proof of Theorem~\ref{thm:exist for bdd range}]
Take $\psi:= \phi-P(\phi)$.
By the remark below \eqref{eq:lmuhmu} and Corollary~\ref{cor:var prin},
we have $P(\phi) = P_+(\phi) = P_G(\hat \phi)$.  Notice that $V_n(\phi)\to 0$ implies that $\beta_n(\hat\phi)=o(n)$ and $\hat\phi$ is continuous in the symbolic metric on $(\hat I, \hat f)$.

Take $\hat X \subset \hat I_{\mbox{\tiny trans}}$ compactly contained in its
domain in the Hofbauer tower and satisfying the statement of
Lemma~\ref{lem:compat entro}.
By Proposition~\ref{prop:exp_tail}, there are $C, \eta >0$ \st
$Z_n^*(\hat\psi,\hat X) <Ce^{-\eta n}$.

We denote the first return time to $\hat X$ by $r_{\hat X}$,
the first return map to $\hat X$ by $R_{\hat X}:=\hat f^{r_{\hat X}}$
and the induced potential by $\hat\Psi:=\psi_{r_{\hat X}}$.
We will shift these potentials, defining $\psi^S:=\psi-S$.
Then $\hat\Psi^S=\Psi-S r_{\hat X}$.
Since $P_G(\hat\psi)=0$ and therefore
$Z_n(\hat \psi,\hat X)< e^{o(n)}$,
we can estimate $Z_0$ from \eqref{eq:Z0} for $S > -\eta$ as
\begin{align*}
Z_0(\hat \Psi^S)  =\sum_n\sum_{r_{\hat X}(x)=n} e^{\hat\psi_n(x)-nS}
& \le \sum_nZ_n^*(\hat\psi-S, \hat X) \\
& \le C\sum_ne^{n(-S-\eta)} Z_n(\hat\psi, \hat X)\\
& \le C'\sum_ne^{n(P_G(\hat\psi) -S - \eta) + o(n)} <\infty.
\end{align*}
Since $P_G(\hat\psi)=0$, this implies that $P_G(\hat\Psi^S)<\infty$ for all $S>-\eta$.
In fact, it also shows that (a) of Proposition~\ref{prop:all_indu} holds.  We let $S^*\le -\eta<0$ be minimal \st $P_G(\hat\Psi^S)<\infty$ for all $S>S^*$.

We can prove precisely the same estimates for the map $F=f^\tau$, where $\tau = r_{\hat X}\circ\pi|_{\hat X}^{-1}$, and the potential $\Phi=\phi_\tau$.  That is, for all $S>S^*$,  $P_G(\Psi^S)<\infty$ and (a) of Proposition~\ref{prop:all_indu} holds.
By Lemma~\ref{lem:choice of X}, item (c) of Proposition~\ref{prop:all_indu}
holds.  Therefore, Case 2 of Proposition~\ref{prop:all_indu} implies that there exists a unique equilibrium state $\mu_\psi$ with $\hat\mu_\psi(\hat X)>0$.

To show that $\mu$ is the unique equilibrium state over $I$,
we assume that there is another equilibrium state $\mu'$.
Let $\hat\mu'$ be the corresponding measure on $\hat I$ from
Theorem~\ref{thm:Kell}.   We now use the fact that $\hat\mu$ is positive on
cylinders.  This follows firstly by the Gibbs properties of the measures
obtained for $(X,F,\mu)$, and then by the transitivity of $(I,f)$
and $(\hat I_{\mbox{\tiny trans}}, \hat f)$.
Thus there exists some cylinder $\hat X'$ in the Hofbauer tower which has
$\hat\mu(\hat X'), \hat\mu'(\hat X')>0$.

We can use the above arguments to say that the corresponding inducing scheme
$(X',F', \Psi')$ satisfies (a) of Proposition~\ref{prop:all_indu}.
But since $\mu_\psi$ is an equilibrium state compatible with $(X', F')$,
also (b) is satisfied.  Therefore, Case 1 of Proposition~\ref{prop:all_indu}
completes the proof of uniqueness.

Finally we note that $\mu_\Psi\{\tau>n\}$ decays exponentially in $n$, since
by the Gibbs property there is $C \ge 1$ \st
\[
\mu_\Psi(\{\tau>n\})= \sum_{\tau_i>n}\mu_\Psi(X_i)
\le C \sum_{\tau_i>n} e^{\Psi_i} = C \sum_{k\ge n} Z_k^*(\psi, \hat X).
\]
By Proposition~\ref{prop:exp_tail}, the latter quantity decays exponentially,
as required.
\end{proof}

\section{Analyticity of the Pressure Function}
\label{sec:ana}

In this section we prove Theorem~\ref{thm:ana nat}.
Throughout, let $\phi_t = -t \log|Df|$.
Let $X \subset I$ and $(X,F,\tau)$ be an inducing scheme on $X$
where $F= f^\tau$. As usual we denote the set of domains of the
inducing scheme by $\{ X_i \}_{i \in \N}$.  Define a tower over the inducing scheme as follows (see \cite{Y})
\[
\Delta = \bigsqcup_{i \in \N} \bigsqcup_{j = 0}^{\tau_i-1}
(X_i,j),
\]
with dynamics
\[
f_\Delta(x,j) = \left\{ \begin{array}{ll}
(x,j+1) & \mbox{ if } x \in X_i, j < \tau_i-1; \\
(F(x),0) & \mbox{ if } x \in X_i, j = \tau_i-1. \\
\end{array} \right.
\]
For $i\in \N$ and  $0\le j<\tau_i$, let $\Delta_{i,j} :=\{(x,j):
x\in X_i\}$ and $\Delta_l:=\bigcup_{i \in \N}\Delta_{i,l}$ is
called the {\em $l$-th floor}. Define the natural projection $\pi_\Delta:\Delta \to X$ by $\pi_\Delta(x,j) =
f^j(x)$.
Note that $(\Delta, f_\Delta)$ is a Markov system, and the first return map of $f_\Delta$ to the \emph{base} $\Delta_0$ is isomorphic $(X,F,\tau)$.

Also, given $\psi:I\to \R$, let $\psi_\Delta:\Delta\to \R$ be defined by $\psi_\Delta(x,j) =\psi(f^j(x))$.  Then the induced potential of $\psi_\Delta$ to the first return map to $\Delta_0$ is exactly the same as the induced potential of $\psi$ to the inducing scheme $(X,F,\tau)$.

The differentiability of the pressure functional can be expressed
using directional derivatives $\left.\frac d{ds}P_G(\psi+s \upsilon)\right|_{s=0}$.  For inducing scheme $(X,F, \tau)$, let $\psi_\Delta$ and $\upsilon_\Delta$ be the lifted potentials to $\Delta$.
Suppose that for $\psi_\Delta:\Delta \to \R$, we have $\beta_n(\psi_\Delta)=o(n)$.  We define the set of \emph{directions} with respect to $\psi$:
\begin{align*}
Dir_F(\psi):=\Bigg\{\upsilon:\sup_{\mu\in \M_+}\left|\int\upsilon~d\mu\right|<&\ \infty,\ \beta_n(\upsilon_\Delta)=o(n),\ \sum_{n=2}^\infty
V_n(\Upsilon)<\infty,\hbox{ and }\\
& \exists \eps>0 \hbox{ s.t. }
P_G(\psi_\Delta+s\upsilon_\Delta)<\infty \ \forall \ s \in
(-\eps,\eps) \Bigg\},
\end{align*}
where $\Upsilon$ is the induced potential of $\upsilon$.
Let $\psi^S := \psi-S$ (and so $\Psi^S=\Psi-S \tau$).  Set $p_F^\ast[\psi]:=\inf\{S : P_G(\Psi^S)<\infty\}$.\footnote{Note that we use
the opposite sign for $p_F^\ast[\psi]$ to Sarig.}  If $p_F^\ast[\psi]>-\infty$, we define the \emph{$X$-discriminant}
of $\psi$ as
\[
\dm_F[\psi]:=\sup\{P_G(\Psi^S):S > p_F^\ast[\psi]\}\le \infty.
\]
Given a dynamical system $(X,F)$, we say that a potential $\Psi:X
\to \R$ is \emph{weakly H\"older continuous} if there exist
$C,\gamma>0$ \st
\begin{equation}\label{eq:wH}
V_n(\Psi)\le C\gamma^n \mbox{ for all } n \ge 0.
\end{equation}

The following is from \cite[Theorem 5]{BTeqnat}.

\begin{theorem} Let $f \in \H$ be a map with potential
$\phi:I\to (-\infty,\infty]$.  Suppose that $\phi$ satisfies
condition \eqref{eq:beta}.
Take $\psi = \phi - P(\phi)$.  Then
$\dm_F[\psi]>0$ if and only if $(X,F,\mu_\Psi)$ has exponential tails.
\label{thm:tails}
\end{theorem}

We are now ready to prove Theorem~\ref{thm:ana nat}.
In this, and the proofs in the sequel, we write
$A_n \asymp B_n$ if $\frac{A_n}{B_n} \to 1$ as $n \to \infty$.
We also write $A\asymp_{dis} B$ if there exists a distortion constant $K\in [1,\infty)$ so that $\frac1K A \le B\le KA$.

\begin{proof}[Proof of Theorem~\ref{thm:ana nat}]
We fix $(X,F)$ as in Proposition~\ref{prop:induce}.
Lemma 5 implies that we have exponential tails for the equilibrium state
associated to the constant potential $\psi=-\htop(f)$, \ie there exist $C,\ \eta>0$ \st
\begin{equation}
\mu_{-\tau \htop(f)}\{\tau_i=n\}\le Ce^{-\eta n}.\label{eq:tail}
\end{equation}
Hence Theorem~\ref{thm:tails} implies that we have positive discriminant.  We can then apply the arguments of the proof of \cite[Theorem 5]{BTeqnat} to show that for $\upsilon\in Dir(-\htop(f))$, there exists $\eps>0$ \st $t\mapsto P(-\htop(f)+t\upsilon)$ is analytic.

Therefore, in order to ensure analyticity here we must prove
$-\log|Df|\in Dir(-\htop(f))$.  It follows from
\cite[Lemma 7]{BTeqnat} that this potential has $\sum_{n=2}^\infty V_n(-\log|DF|)<\infty$, and \cite{Prz} gives $\sup_{\mu\in \M_+}\left|\int\log |Df|~d\mu\right|< \infty$; so it only remains to prove that there exists $\eps>0$ \st
$P_G((-\htop(f)-t\log|Df|)_\Delta)<\infty$ for $t\in (-\eps, \eps)$.
Since $P_G((-\htop(f)-t\log|Df|)_\Delta)\le P_G(-\tau\htop(f)-t\log|DF|)$,
by Abramov´s Theorem it suffices to bound $P_G(-\tau\htop(f)-t\log|DF|)$.
As in Section~\ref{subsec:pressure}, $Z_0(\Phi)<\infty$ implies $P_G(\Phi)<\infty$.  In the following calculation we use the fact that for all $\eps>0$ there exists $C_\eps>0$ so that  $\#\{\tau_i=n\}\le C_\eps e^{n(\htop(f)+\eps)}$, see the discussion at \eqref{eq:lapsfn}.  For $0<t<1$, choose $0<\eps<\left(\frac{t}{1-t}\right)\htop(f)$.  Using the H\"older inequality,
\begin{align*}
Z_0(-\tau\htop(f)-t\log|DF|) & \asymp_{dis} \sum_ne^{-n\htop(f)}\sum_{\tau_i=n}e^{-t\log|DF_i|} \\
&\asymp_{dis} \sum_ne^{-n\htop(f)}\sum_{\tau_i=n}|X_i|^t\\
&\le  \sum_ne^{-n\htop(f)}\left(\sum_{\tau_i=n}|X_i|\right)^t(\#\{\tau_i=n\})^{1-t}\\
&\le C_\eps^{1-t}\sum_n e^{n(-\htop(f)+(1-t)(\htop(f)+\eps))}\\
& =  C_\eps^{1-t}\sum_n e^{n(-t\htop(f)+(1-t)\eps)}<\infty.
\end{align*}
(For further explanation of these calculations see \cite[Section 5 ]{BTeqnat}.)

For $t<0$, first notice that by the Gibbs property of $\mu_{-\tau \htop(f)}$
$$
\mu_{-\tau \htop(f)}\{\tau=n\} \asymp e^{-n\htop(f)}\sum_{\tau_i=n} 1 = e^{-n\htop(f)}\#\{\tau_i=n\}.
$$
Hence, by \eqref{eq:tail},
\begin{equation}
e^{-n\htop(f)}\#\{\tau_i=n\} \le Ce^{-\eta n}.\label{eq:max ent Gibbs}
\end{equation}
Since $|X_i|\ge |X|e^{-\gamma\tau_i}$ for $\gamma := \log\sup|Df|$, we have
\begin{align*}
Z_0(-\tau\htop(f)-t\log|DF|) & \asymp_{dis} \frac1{|X|^t}\sum_ne^{-n\htop(f)}\sum_{\tau_i=n}|X_i|^t\\
&\le  \sum_n[e^{-n\htop(f)} \#\{\tau_i=n\}] e^{-\gamma n t}\\
&\le  C\sum_ne^{-n(t\gamma+\eta)}<\infty,
\end{align*}
if $t\gamma+\eta>0$. 
Hence there exists $\eps>0$ so that $-\log|Df|\in Dir(-\htop(f)-\eps)$.

It remains to show existence and uniqueness of equilibrium states.
By \eqref{eq:max ent Gibbs}, we have for $t\ge 0$, using the
H\"older inequality again,
\begin{align*}
Z_0(-t\log|DF|-\tau P(\phi_t)) & \asymp_{dis} \sum_ne^{-nP(\phi_t)}\sum_{\tau_i=n}e^{-t\log|DF_i|} \asymp_{dis} \sum_ne^{-nP(\phi_t)}\sum_{\tau_i=n}|X_i|^t\\
&\le \sum_ne^{-nP(\phi_t)}\left(\sum_{\tau_i=n}|X_i|\right)^t\#\{\tau_i=n\}^{1-t} \\
&\le \sum_n [e^{-n\htop(f)}\#\{\tau_i=n\}]^{1-t} e^{-n(P(\phi_t)-(1-t)\htop(f))}\\ &\le C\sum_n e^{n((1-t)(\htop(f)-\eta) -P(\phi_t))}.
\end{align*}
Since $P(\phi_t) \to \htop(f)$ as $t\to 0$, for all small $t$ we have $(1-t)(\htop(f)-\eta') -P(\phi_t)<0$.  Hence $Z_0(-t\log|DF|-\tau P(\phi_t))<\infty$ for small positive $t$.

For $t<0$, we use a similar computation as before:
\begin{align*}
Z_0(-t\log|DF|-\tau P(\phi_t)) & \asymp_{dis} \sum_ne^{-nP(\phi_t)}\sum_{\tau_i=n}e^{-t\log|DF_i|}\\
& \asymp_{dis} \sum_ne^{-nP(\phi_t)}\sum_{\tau_i=n}|X_i|^t\\
& \le \sum_ne^{-n(P(\phi_t)+t\gamma)} \#\{\tau_i=n\}\\
& <C_\eps \sum_ne^{-n(t\gamma+P(\phi_t)-\htop(f)-\eps)},
\end{align*}
where we use the fact that  for all $\eps>0$ there exists $C_\eps>0$ so that $\#\{\tau_i=n\}\le C_\eps e^{n(\htop(f)+\eps)}$.  Since $P(\phi_t)>\htop(f)$ we can ensure that $t\gamma+P(\phi_t)-\htop(f)-\eps>0$ for all $t$ close to zero.  Hence $Z_0(-t\log|DF|-\tau P(\phi_t))$ is finite for all $t$ close enough to zero.

This implies that for $t$ in a neighbourhood of 0,
$P_G(-t\log|DF|-\tau P(\phi_t))<\infty$.  Similarly property (a) of Proposition~\ref{prop:all_indu} holds, and thus we can apply Case 2 of that proposition to get existence of an equilibrium state $\mu$.  This is the unique equilibrium state among those that can be lifted to $(X,F)$.  Following the argument in the proof of Theorem~\ref{thm:exist for bdd range}, we have that $\mu$ is the unique global equilibrium state as required.
\end{proof}

\section{Necessity of the Condition $\sup\phi-\inf\phi<\htop(f)$}

\label{sec:examples}

In this section we show the importance of the condition \eqref{eq:bdd range} for the existence and uniqueness of equilibrium states obtained by inducing methods.

Hofbauer and Keller gave an example, originally in a symbolic setting
\cite{Hnonuni}
and later in the context of the angle doubling map on the circle  \cite{HKeq},
which showed that \eqref{eq:bdd range} is essential for their results on
quasi-compactness of the transfer operator.
In Section~\ref{sec:HK}, we discuss how that example fits in with our
inducing results.  The Hofbauer and Keller example uses a non-H\"older
potential, so it is natural to ask if is really the lack of H\"older
regularity which causes problems in obtaining equilibrium states.
In Section~\ref{sec:M-P}, we provide an example of a family of H\"older
continuous potentials which, if a member of the family  violates
\eqref{eq:bdd range}, then the equilibrium state is not obtained from any
inducing scheme with integrable inducing time.

We note here that these Markov examples are often modelled by the renewal
shift, see \cite{Saphase} and \cite{PeZphase}.
That approach uses a rather different partition to the one we use in this
paper, and so does not elucidate our theory.  However, the inducing schemes
we use and the ones that \cite{Saphase} and \cite{PeZphase} get from the
renewal shift are the same.

\subsection{Hofbauer and Keller's Example}
\label{sec:HK}

As mentioned in Theorem~\ref{thm:HK}, potentials $\phi\in BV$
satisfying
$\sup\phi-\inf\phi<\htop(f)$
have equilibrium states; in fact
Hofbauer and Keller \cite{HKeq} show that this equilibrium state is
absolutely continuous w.r.t. to a $\phi$-conformal measure, and that the
transfer operator is quasi-compact.
They also present, for the angle doubling map $f(x) = 2x \pmod 1$, a class
of potentials $\phi$ to show that \eqref{eq:bdd range} is essential for
these latter properties.
This map $f$ was inspired by an example based in \cite{Hnonuni}
based on the full shift $\sigma: \{ 0,1\}^{\N} \to   \{ 0,1\}^{\N}$,
showing that H\"olderness of potentials is essential
to obtain the results from \cite{Bowen}.

We demonstrate how this class of examples fits into the framework of our paper.
Fix $K \ge 0$ and let $b<0$.  Let
$$
\phi = \phi_{b,K} = \sum_{k=0}^\infty a_k\cdot
1_{(2^{-k-1},2^{-k}]},
$$
where
\begin{equation*}
a_k:=\left\{\begin{array}{cl} b &
\hbox{ for } 0\le k< K,\\
2\log\left(\frac{k+1}{k+2}\right) & \hbox{ for } k \ge K.
\end{array} \right.
\end{equation*}
Also let $s_n = \sum_{k=0}^{n-1} a_k$.
Since the Dirac measure $\delta_0$ at the fixed point has free
energy $h_{\delta_0}(f) + \phi(0) = 0$, the pressure $P(\phi) \ge 0$.
Figure~\ref{fig1} summarises the
results of \cite{Hnonuni} and the example in \cite{HKeq} that are
relevant for us.

\begin{figure}[ht]
\unitlength=5mm
\begin{picture}(20,12)(3,0)
\thinlines
\put(0,0){\line(1,0){26}}
\put(1,0.8){$\sum_k e^{s_k} < 1$}
\put(0,2){\line(1,0){26}}
\put(1,3.3){$\sum_k e^{s_k} = 1$}
\put(0,5){\line(1,0){26}}
\put(1,6.3){$\sum_k e^{s_k} > 1$}
\put(0,8){\line(1,0){26}}

\put(5,0){\line(0,1){8}}
\put(0,0){\line(0,1){8}}

\put(5.5,2.5){\small $\sum_k (k+1)e^{s_k} = \infty$}
\put(5,3.5){\line(1,0){21}}
\put(5.5,4){\small $\sum_k (k+1)e^{s_k} < \infty$}
\put(5.5,5.5){\small $\sum_k a_k = \infty$}
\put(5,6.5){\line(1,0){21}}
\put(5.5,7){\small $\sum_k a_k < \infty$}

\put(12,0){\line(0,1){12}}
\put(17,0){\line(0,1){12}}
\put(21,0){\line(0,1){12}}
\put(26,0){\line(0,1){12}}

\put(13,11){\small }
\put(13,10){\small Pressure}
\put(13,9){\small $P(\phi)$}

\put(13,7){$P(\phi) > 0$}
\put(13,5.5){$P(\phi) > 0$}
\put(13,4){$P(\phi) = 0$}
\put(13,2.5){$P(\phi) = 0$}
\put(13,1){$P(\phi) = 0$}

\put(17.8,11){\small $\mu_\phi$ is}
\put(17.8,10){\small a Gibbs}
\put(17.8,9){\small measure}

\put(18.5,7){\small yes}
\put(18.5,5.5){\small no}
\put(18.5,4){\small no}
\put(18.5,2.5){\small no}
\put(18.5,1){\small no}

\put(21.5,11){\small $\phi$ has a}
\put(21.5,10){\small unique equi-}
\put(21.5,9){\small librium state}

\put(23,7){\small yes}
\put(23,5.5){\small yes}
\put(23,4){\small no}
\put(23,2.5){\small yes}
\put(23,1){\small yes}

\put(12,12){\line(1,0){14}}
\end{picture}
\caption{Summary of results in \cite{Hnonuni}: Equation (2.6) and Section 5.}
\label{fig1}
\end{figure}

Define the inducing scheme $(X,F)$ where $X=(\frac12,1]$ and
$F:\bigcup_nX_n\to X$ is the first return map to $X$ where for $n\ge 1$, $X_n:=\left(\frac12+2^{-n-1}, \frac12+2^{-n}\right]$.
Notice that if we denote $X^\infty = \{ x : \#\orb(x) \cap X = \infty\}$,
then $\mu(X^\infty) = 1$ for every measure in
$\M_{erg} \setminus \{ \delta_0\}$.

In \cite{HKeq}, it is important that $b$ is chosen so that
$-b>\htop(f)=\log 2$, but for our case we allow $b$ to vary.
\begin{lemma} For all $K\ge 2$ there exists $b_K<-\log 2$ \st
\begin{itemize}
\item $b>b_K$ implies $P(\phi_{b,K})>0$ and there exists a unique equilibrium state which can be found from $(X,F)$;
\item $b\le b_K$ implies $P(\phi_{b,K})=0$ and the unique equilibrium state is the Dirac measure $\delta_0$ on $0$.  This cannot be found from $(X,F)$.
\end{itemize}
Moreover, $b_K \to -\log 2$ as $K \to \infty$.
\label{lem:HK}
\end{lemma}

\begin{proof}
Firstly, we compute
\begin{equation*}\label{eq:s_n} s_n=  \left\{\begin{array}{ll}
nb & \hbox{ if } n\le K, \\
Kb+2\log\left(\prod_{j=K}^{n-1}\left(\frac{j+1}{j+2}\right)\right)= Kb+2\log\left(\frac{K+1}{n+1}\right) & \hbox{ if } n>K. \end{array} \right.
\end{equation*}
As in \cite{HKeq}, we can estimate
\begin{align}\label{eq:sumsn}
\sum_n e^{s_n}& =\sum_{n=1}^Ke^{nb}+
e^{Kb}\sum_{n>K}\left(\frac{K+1}{n+1}\right)^2<
e^b\left(\frac{1-e^{bK}}{1-e^b}\right)+e^{bK}(K+1).
\end{align}
For $b<-\log2$ the first term is strictly less than 1 for all $K$ and the
second term tends to zero as $b\to -\infty$.  Hence if we fix $K$,
then we can find $b_K$ such that $\sum_n e^{s_n} \le 1$
for $b \le b_K$ (with equality if and only if $b = b_K$),
and Figure~\ref{fig1} shows that $P(\phi)=0$.
Alternatively, by fixing $b<-\log2$ and taking $K$ large enough we have
$P(\phi)=0$, and in fact $b_K\to -\log 2$ as $K\to \infty$.
A computation similar to \eqref{eq:sumsn} shows that
$\sum_n (n+1) e^{s_n} \ge C \sum_{n > K} (n+1) (\frac{K+1}{n+1})^2$
diverges.  Whenever $P(\phi)=0$, Figure~\ref{fig1} shows that
$\delta_0$ is the unique equilibrium state.

We next show what $P(\phi)=0$ or $P(\phi)>0$ imply for obtaining the
equilibrium state from the inducing scheme.
As usual, we set $\psi:=\phi-P_G(\phi)$.
Notice that $V_n(\Psi)=0$, so clearly we have summable variations.
Also notice that for $x\in X_n$,
$$
\Psi(x)= s_n-nP(\phi) \asymp -nP(\phi)+2\log\left(\prod_{k=0}^{n-1}\left(\frac{k+1}{k+2}\right)\right) = -nP(\phi)-2\log(n+1).
$$
Therefore
\begin{align}
Z_0(\Psi) & = \sum_{n=1}^\infty e^{\Psi|_{X_n}} \asymp_{dis}
\sum_{n=0}^\infty e^{-nP(\phi)-2\log(n+1)}
=\sum_{n=0}^\infty \frac{e^{-nP(\phi)}}{(n+1)^2}<\infty, \label{eq:sn}
\end{align}
because $P(\phi)\ge 0$.  So as in Section~\ref{subsec:pressure} this means that
$P_G(\Psi)<\infty$.  Thus Theorem~\ref{thm:BIP} yields a Gibbs state $\mu_\Psi$.  Similarly to the calculation above, we can show from the Gibbs property of $\mu_\Psi$ that $$-\int\Psi~d\mu_\Psi\asymp_{dis}
\sum_{n=0}^\infty\frac{e^{-nP(\phi)}\log(n+1)}{(n+1)^2}<\infty.
$$ for $P(\phi)\ge 0$.
Therefore, $\mu_\Psi$ is an equilibrium state for $(X, F)$.  We also have
\begin{equation}\label{eq:infind}
\int\tau~d\mu_\Psi\asymp_{dis} \sum_{n=1}^\infty \frac{ne^{-nP(\phi)}}{(n+1)^2} \qquad  \left\{\begin{array}{ll} < \infty & \hbox{ if } P(\phi)>0, \\
=\infty & \hbox{ if } P(\phi)=0. \end{array} \right.
\end{equation}
Therefore if $P(\phi)=0$, we cannot project this measure to the
original system.
\end{proof}

In the limit $K \to \infty$, the potential is $\phi(x)=b$ for $x\in (0,1]$
and $\phi(0)=0$.  It is easy to see that the same results above hold in this
case 
and that for $\phi_{-\log 2, \infty}$ the
equilibrium states are $\delta_0$ and the measure of maximal entropy.

We briefly summarise the conclusions of this example,
in order to clarify how it fits in with the results stated in this paper.
We fix $K\ge 2$.
Since $\phi$ is monotone, $\| \phi \|_{BV} < \infty$, but
$\sum_n \sup_{\c \in \P_n} \| \phi|_\c \|_{BV} = \sum_n V_n(\phi) = \infty$.
Therefore
Theorem~\ref{thm:Pac} does not apply for any value of $b$.
\begin{itemize}
\item For $b\le b_K$, we have $P(\phi)=0$ but \eqref{eq:bdd range} fails,
so Theorems~\ref{thm:HK} and \ref{thm:exist for bdd range} do not apply.
However, there exists a unique equilibrium state $\delta_0$ by \cite{Hnonuni}.
\item For $b_K<b\le -\log 2$, we have $P(\phi)>0$, but again
Theorems~\ref{thm:HK} and \ref{thm:exist for bdd range} do not apply.
However, there exists a unique equilibrium state by \cite{Hnonuni}.  Moreover, direct computations as in \eqref{eq:sn} and \eqref{eq:infind} allow us to use our inducing method and Case 2 of Proposition~\ref{prop:all_indu} to show that there exists a unique equilibrium state, which can be obtained from an inducing scheme.
\item For $-\log 2<b<0$, Theorem~\ref{thm:HK} applies
(since $\| \phi \|_{BV} < \infty$)  and Theorem~\ref{thm:exist for bdd range}
applies because
$\Psi$ is piecewise constant (so \eqref{SVI} holds and in fact,
$\Psi$ is weakly H\"older continuous,
see \eqref{eq:wH}).
Both theorems produce the unique equilibrium state.
\end{itemize}

In general, inducing schemes are used to improve the hyperbolicity of the map or properties of the potential (e.g.\ to obtain weak H\"older continuity).
For this system (or for the Manneville-Pomeau map of
Section~\ref{sec:M-P} below), there are inducing schemes that
produce the equilibrium state $\delta_0$. For instance, one can take
the original map itself, or the `unnatural' system consisting of the
left branch only, as induced system. But to obtain nice properties for map
or potential, one has to induce to a domain disjoint from $0$,
and none of these `natural' inducing schemes produces $\delta_0$ as
equilibrium state.

For $b\le b_K$ we have $\dm_F[\phi]=0$, since $P_G(\Phi-S\tau)=\infty$ for all $S<0$.
If $\phi$ had summable variations, then
the discriminant theorem \cite{Saphase} would imply
that $\phi$ is not `strong positive recurrent', but can be either positive recurrent or null recurrent.  The fact that we cannot project $\mu_\Psi$ appears to suggest that $\phi$ is null recurrent.  However, since the variations of $\phi$ are not summable we are not able to use this theory.  However, in the following lemma we make a direct computation to show that indeed $\phi$ is null recurrent when $b \le b_K$.

\begin{lemma}
Fix $K\ge 2$.  If $b\le b_K$ then $\phi$ is null recurrent.
\label{lem:phase}
\end{lemma}

\begin{proof}
Let $\c_0$ and $\c_1$ the left and right cylinders in $\P_1$.
Rather than considering all $n$-periodic cycles, we will restrict
ourselves to special ones, and show that these are sufficient to imply
recurrence.
For each $n$ there is a cycle $\cyc_n:=\{p_n^n, \ldots, p_n^1\}$ where $p_n^1\in  X_n$ as defined above, $f(p_n^k)=p_n^{k-1}$ for $n\ge k\ge 2$ and
$f(p_n^1)=p_n^n$  (in fact it is easy to compute $p_k^n=\frac{2^{n-k}}{2^n-1}$).
For $x\in \cyc_n$,  $\phi_n(x)=s_n$.
This cycle features $n-1$ times in the computation of
$Z_n(\phi,\c_0)$.  Hence,
$$
Z_n(\phi,\c_0)\ge ne^{s_n} \ge
(n-1-K) \left[\left(\frac K{n+1}\right)^2 \cdot e^{Kb}\right],
$$
so $\sum_n Z_n(\phi,\c_0) \ge \sum_n \frac{C}{n} = \infty$.
Recalling that $P_G(\phi) = 0$ for $b \le b_K$, this implies that
the potential is recurrent.

Notice that $p^1_n$ is the only point in $\cyc_n$ that belongs to
$\c_1$. So using this point and cylinder $\c_1$, the same computation
implies
that $\sum_n n Z^*_n(\phi, \c_1) = \infty$, so $\phi$ is null recurrent.
\end{proof}

\subsection{The Manneville-Pomeau Map}
\label{sec:M-P}

The Manneville-Pomeau map $f_{\alpha}(x) = x + x^{1+\alpha} \pmod 1$
with $\alpha \in (0,1)$ is well-known to have zero entropy equilibrium states
for the potential $-t \log |Df_{\alpha}|$ and appropriate values of $t$.
See \cite{Saphase} for an exposition of this theory and  the relevant references.
Supposing that $\alpha<\frac{\log 2}2$, for $p_1<p_2<1$ and $b<-\log 2$ we will use the potential
$$
\phi(x) = \phi_{\alpha, p_1,p_2,b}(x):= \left\{\begin{array}{ll}
-2\alpha x^{\alpha} &\hbox{ if } x\in [0,p_1],\\[1mm]
\left(\frac{b+2\alpha p_1^\alpha}{p_2-p_1}\right)(x-p_1)-2\alpha p_1^\alpha &\hbox{ if } x\in (p_1,p_2],\\[1mm]
b &\hbox{ if } x\in (p_2,1], \end{array}\right.
$$
as an example to show that \eqref{eq:bdd range}
is sharp.
(Note that $\phi$ has the same H\"older exponent as $-\log |Df_{\alpha}|$.)
Since $\htop(f) = \log 2$,
condition  \eqref{eq:bdd range} is violated whenever $b \le -\log 2$.
It turns out that as soon as this occurs, we can choose $\alpha, p_1, p_2$ so that no equilibrium state can be achieved
from a `natural' inducing scheme on an interval bounded away from the neutral fixed point $0$.
Thus  \eqref{eq:bdd range}
is sharp, even when the potential is H\"older.

The conclusion of Proposition~\ref{prop:MP} proved below is that
H\"older regularity  of the potential is not sufficient to dispense with the condition \eqref{eq:bdd range}.

\begin{proof}[Proof of Proposition~\ref{prop:MP}]
We will make a suitable choice for $p_1, p_2$ later in the proof.  Let $y_0=1$ and define $y_n\in (0,y_{n-1})$ for $n\ge 1$ such that
$f_\alpha(y_{n})=y_{n-1}$.
From the recursive relation $y_n = y_{n+1}(1+y_{n+1}^\alpha)$ we derive
(cf. \cite{Bruyn})
\[
\frac1{y_n} = \frac1{y_{n+1}} ( 1+y_{n+1}^\alpha)^{-1}
= \frac1{y_{n+1}} \left( 1-y_{n+1}^\alpha + y_{n+1}^{2\alpha} +Err(y_{n+1}^{3\alpha}) \right),
\]
where $|Err(y_{n+1}^{3\alpha})|=O(y_{n+1}^{3\alpha})$.  Using $u_n = y_n^{-\alpha}$ this becomes
\begin{eqnarray*}
u_n &=& u_{n+1}\left(1-\frac1{u_{n+1}} +
\frac1{u_{n+1}^2}+Err\left(\frac1{u_{n+1}^3}\right)\right)^\alpha \\
&=&  u_{n+1} \left(1-\frac{\alpha}{u_{n+1}} +  \frac{\alpha(\alpha+1)}{2u_{n+1}^2} + Err\left(\frac1{u_{n+1}^3}\right) \right),
\end{eqnarray*}
where $\left|Err\left(\frac1{u_{n+1}^3}\right)\right|=
O\left(\frac1{u_{n+1}^3}\right)$.
Therefore $u_{n+1}-u_n = \alpha + \frac{ \alpha(\alpha+1) }{2}
\frac{1}{u_{n+1}} + Err\left(u_{n+1}^{-2}\right)$, and using telescoping series this leads to
\[
u_n = \alpha n + \frac{ \alpha(\alpha+1) }{2} \log n + Err\left(\frac1n\right) \
\]
Transforming back to the original coordinate $y_n$, we find
\[
y_n = \left(\frac{1}{u_n}\right)^{1/\alpha} =
\left(\frac1\alpha\right)^{1/\alpha} \left(\frac1n\right)^{1/\alpha} \left( 1 + \frac{\alpha(\alpha+1)}{2n} \log n + Err\left(\frac1{n^2}\right) \ \right)^{-1/\alpha}.
\]
Thus
\begin{eqnarray*}
\phi(y_n) = -2\alpha y_n^{\alpha} & = &
\frac{-2}{n}\left(  1 + \frac{\alpha(\alpha+1)}{2n} \log n +
Err(n^{-2}) \ \right)^{-1} \\
&=&
\frac{-2}{n} + \frac{\alpha(\alpha+1)}{n^2} \log n +
Err(n^{-3})
\end{eqnarray*}
for $y_n < p_1$ where $|Err(n^{-3})|=O(n^{-3})$.

For all $n$ sufficiently large, the variations w.r.t. the branch partition satisfy $V_n(\phi) \ge \frac{2}{n}$
(obtained on the $n$-cylinder
set $[0, y_n]$), so $\phi$ does not have summable variations.
However, since $\phi$ is monotone, $\|\phi\|_{BV} < \infty$.

For any $b<-\log 2$ we will choose $K>N$ and $p_1=K$ and $p_2=y_N$ depending on $\alpha$ and $b$.
$n\ge N$ implies (replacing the convergent sum of the last given and
higher order terms by a single constant $B = B_N$ which is bounded in $N$),
$$
s_n := \sup_{x \in (y_{n+1}, y_n]} \ \sum_{k=0}^{n-1} \phi(f^k(x)) = Nb  -2\sum_{k=N}^{n-1} \frac1k +  B.
$$
Clearly choosing $N$ large enough we can make this error
as small as we like.  By the above, we have
\begin{align*}
\sum_ne^{s_n} & \le \sum_{k=1}^Ne^{kb}+e^{Nb+B}\sum_{N+1}^\infty \left(\frac Nn\right)^2
\le  e^b\left(\frac{1-e^{Nb}}{1-e^b}\right)+e^{Nb+B}(N+1).
\end{align*}
Hence, we can choose $N$ so large that $\sum_n e^{s_n}\le 1$ and hence by Figure~\ref{fig1} we have $P(\phi)=0$.
(Likewise we can fix suitable $\alpha, N, K$ and find a critical value $b_{\alpha, N, K}$ where below this value, $\sum_n e^{s_n}\le 1$ and above it, $\sum_n e^{s_n}> 1$.)

We define $F$ to be the first return map to $X:=(y_1,1]$,
so if $x_i \in (y_1,1]$ is such that $f_\alpha(x_i) = y_i$,
then $X_i = (x_{i+1}, x_i]$ and $\tau_i = i$.
A straightforward computation shows that $\Phi|_{X_n}$ is monotone and
there is $C \ge 1$ \st for large $n$,
$-2\log n - \frac{C}{n} \le \Phi|_{X_n} \le -2 \log n + \frac{C}{n}$;
in fact $\Phi$ is weakly H\"older.
As in Lemma~\ref{lem:HK}, we can show that $Z_0(\Phi) < \infty$, so
$P_G(\Phi)<\infty$ and there is a unique equilibrium state
$\mu_{\Phi}$ for $(X,F,\Phi)$
which also satisfies the Gibbs property.
However, as in \eqref{eq:infind}, the inducing time has $\int \tau \ d\mu_{\Phi} = \infty$, as required.
\end{proof}

\section{Recurrence of Potentials}\label{sec:recurrence}

Although not crucial for the main results of this paper,
the question whether the potential is recurrent (see \eqref{eq:recurrence})
is of independent interest.
In this section we give sufficient conditions for $\phi$ to be
recurrent, and for the topological pressure and the Gurevich pressure
to coincide.

Recall that Theorems~\ref{thm:HK} and \ref{thm:HKsum} gave conditions
under which transfer operator $\L_\phi$ is quasi-compact.
Let us first lay out an argument why this implies that
$\phi$ is recurrent.
Recall that quasi-compactness means that the essential spectrum $\sigma_{ess}$ is
strictly less than the leading eigenvalue $\lambda = \exp(P(\phi))$, and
there are only finitely many eigenvalues outside $\{|z| \le \sigma_{ess} \}$,
each with finite multiplicity.
A result due to
Baladi and Keller \cite{BalKel} says that this spectral gap
implies that the dynamical $\zeta$-function
\[
\zeta(z) = \exp\left(
\sum_{n=1}^{\infty} \frac{z^n}{n} \sum_{f^n(x)=x} e^{\phi_n(x)} \right)
\]
is meromorphic on $\{ |z| \le \lambda^{-1} \}$, with a pole at
$\lambda^{-1}$ whose multiplicity is the same as the multiplicity of
the eigenvalue $\lambda$ of $\L_\phi$.
The argument why this implies recurrence of the potential is
somewhat implicit in  \cite{BalKel}.
Namely, there is a function $g$ which is analytic on
$\{ |z| < \lambda^{-1} \} \cap \{ |z-\lambda^{-1}| < \eps\}$
such that $g(\lambda^{-1}) \neq 0$ and $\zeta'(z)/\zeta(z) = g(z)/(z-\lambda^{-1})$ on this region.
Hence $\lim_{z \to \lambda^{-1}} \zeta'(z)/\zeta(z) = \infty$.
Direct computation gives
\[
\frac{\zeta'(z)}{\zeta(z)} = \frac1{z} \sum_{n=1}^{\infty} z^n
\sum_{f^n(x)=x} e^{\phi_n(x)}
= \frac1z \sum_{n=1}^{\infty} z^n \ Z_n(\phi),
\]
so recurrence follows.

\begin{proposition}
Let $f\in \H$ and $\phi$ be a potential \st $\sup \phi - \inf \phi < \htop(f)$ .
If
\begin{equation*}\label{eq:cond_thm}
V_n(\phi) \to 0 \quad \mbox{ and } \quad \sum_n e^{-\beta_n} = \infty,
\end{equation*}
then $\phi$ is recurrent. (Here $\beta_n := \beta_n(\phi)$ is defined as in \eqref{eq:beta}.)
\label{prop:rec}
\end{proposition}

Clearly $\beta_n \le \sum_{k=1}^n V_k(\phi)$,
and $V_n(\phi) \to 0$ implies $\beta_n = o(n)$.
The condition $\sum_n e^{-\beta_n} = \infty$ is stronger: it implies
that $\beta_n = o(\log n)$ and is implied by $V_n(\phi) = O(n^{-(1+\eps)})$.

It is well known that the Variational Principle holds for the potential
$\phi = 0$; in fact
\[
\htop(f) = P(0) = P_G(0) = P_{top}(0) = \lim_n \frac1n \log
\laps(f^n),
\]
where $\laps(f^n) := \# \P_n$ is  the \emph{lap number},
\ie the number of maximal intervals on which
$f^n$ is monotone, see \cite{MSz}.
In fact, the lap number is submultiplicative: $\laps(f^{n+m}) \le
\laps(f^n) \laps(f^m)$.
Therefore $\htop(f) = \inf_{n} \frac1n \log \laps(f^n)$ and
\begin{eqnarray}\label{eq:lapsfn}
e^{\htop(f) n} \le  \laps(f^n) \le e^{n(\htop(f)+\eps_n)},
\end{eqnarray}
where $\eps_n\to 0$ as $n \to \infty$.  We will extend this idea to ergodic averages of more general potentials in
Lemma~\ref{lem:Zn}.
For $J \in \P_m$, let $\phi_m(J) = \sup \{ \phi_m(x) : x \in J\}$ and
\[
Z^{top}_m(\phi) := \sum_{J \in \P_m} e^{\phi_m(J)}.
\]
For the remainder of this section we assume that $(I,f)$ is
topologically mixing, \ie for each $m$, $(I, f^m)$ is
topologically transitive.
In order to prove recurrence of $\phi$, we need the following lemma.

\begin{lemma}\label{lem:Zn}
Let $\phi$ be a potential satisfying \eqref{eq:bdd range} and with  $\beta_n(\phi) = o(n)$.
Then there exists $\eta > 0$ \st
$Z_n(\phi) \ge \eta e^{-\beta_n} e^{P_{top}(\phi)n}$ for all $n$,
and $P_{top}(\phi) = P_G(\phi)$.
\end{lemma}

\begin{proof}
Since $f$ is topologically transitive, there is a collection of intervals
permuted cyclically by $f$,
such that for any interval $J$, there is $n$ such that $f^n(J)$ contains
a component of this cycle. For simplicity, let us assume that
this collection is just a single interval $I$.

Since every $m$-cylinder set can contain at most one $m$-periodic point,
$Z_m(\phi) \le Z_m^{top}(\phi)$ for all $m$.
Furthermore, $Z^{top}_m(\phi)$ is submultiplicative, cf.
\eqref{eq:lapsfn}, so
\begin{equation*}\label{eq:Psharp}
P_{top}(\phi) := \lim_{m\to \infty} \frac1m \log Z^{top}_m(\phi)
= \inf_m \frac1m \log Z^{top}_m(\phi) < \infty.
\end{equation*}
Therefore $P_G(\phi) \le P_{top}(\phi) < \infty$.

Recall that every $J \in \P_m$ corresponds to a unique $m$-path $D_0 \to D_1 \to
\dots \to D_m$ in the Hofbauer tower $(\hat I, \hat f)$ leading
from the base $D_0$ of the tower to some terminal domain
$D_m$. The level of $D_m$ was defined as the length of the shortest path from the base to $D_m$.
We say that the \emph{pre-level} of $J$ is $\plev(J) = R$.

The topological entropy $\htop(f)$ is the exponential growth rate of
the number of $n$-paths $D_0 \to \dots \to D_n$ in the Hofbauer
tower, and the limit of the exponential growth rates of the number
of $n$-paths within $\hat I_R$ as $R \to \infty$, see \cite{Hpwise}
for the unimodal and \cite[Sections 9.3-9.4]{BB} for the general
case. Therefore, by taking $R$ sufficiently large, we can find $\gamma >
0$ and $C_0 \in (0,1)$ \st the number of $k$-paths
\begin{equation}\label{eq:R1}
\#\{ D_0 \to D_1 \to \dots \to D_k \ : \ \lev(D_k) \le R, 1 \le j
\le k \} \ge C_0 e^{ k (\htop(f)-\gamma ) }
\end{equation}
for all $k \ge 1$, and
\begin{equation}\label{eq:R2}
\sup \phi - \inf \phi < \htop(f) - \gamma - \frac{\log 2}{R}.
\end{equation}
Since $(I,f)$ is topologically transitive
(and using our simplifying assumption), there exists $R'$
depending on $R$, \st for each $D \in \D$ with $\lev(D) \le R$,
$f^{R'}(D) \supset I$. This implies that every $J \in \P_m$ with $\plev(J) \le R$.
contains a periodic point of period $n := m+R'$.

The idea is now for an arbitrary $J \in \P_m$ to extend the corresponding
path by $R'$ arrows to find an $n$-periodic point $p \in J$.
If $\plev(J) \le R$, then by the choice of $R'$, this is indeed possible.
We call such cylinder sets $J$ {\em type 1}, and we can thus compare
$Z_m^{\mbox{\tiny type 1}}(\phi)$ to $Z_n(\phi)$
as:
\begin{eqnarray}\label{eq:type1}
Z_m^{\mbox{\tiny type 1}}(\phi)
&=& \sum_{J \in \P_m, \mbox{\tiny type 1}} e^{\phi_m(J)}  \nonumber \\
&\le& \sum_{\stackrel{p = f^n(p) \in J}{J \in \P_m\ \mbox{\tiny is type 1}} }
e^{\beta_m} e^{-R' \inf \phi} e^{\phi_n(p)}
\le e^{\beta_m} e^{-R' \inf \phi} Z_n(\phi).
\end{eqnarray}
If $\plev(J) > R$, then the existence of an $n$-periodic point in $J$
cannot be guaranteed. We call such cylinder sets $J$ {\em type 2}.
Given such a type 2 cylinder set $J$, there is
a maximal $m'<m$ such that $\plev(J') = R$ for the $m'$-cylinder
$J'$ containing $J$.
As we mentioned before, from any domain in the Hofbauer tower, there are
at most two $R$-paths that are outside $\hat I_R$.
Using this property repeatedly, we find that there are at
most $2^{(m-m')/R} = e^{(m-m')\frac{\log 2}{R}}$ starting at $D_{m'}$
but otherwise outside $\hat I_R$.
From $D_{m'}$, there is at least one $R'$-path leading back to some $D \in \hat I_R$, and using \eqref{eq:R2} and \eqref{eq:R1} we derive that there are
at
least $C_0 e^{(m-m'-R')(\htop(f)-\gamma)}$ `type 1' $m-m'$-paths from $D_{m'}$.
From this we conclude that the type 1 cylinders
``sufficiently'' outnumber the type 2 cylinders,
and we can bound the contributions of type 2 cylinders in $J'$
by the contribution of type 1 cylinders in $J'$ as follows:
\begin{align*}
\sum_{J \subset J', \mbox{\tiny type 2}} e^{\phi_m(J)}
&\le e^{(m-m')(\sup \phi + \frac{\log 2}{R})} \\
&\quad \times  \frac1{C_0} \
e^{-(m-m'-R')(\htop(f) - \gamma)}
e^{-(m-m') \inf \phi}
\sum_{J \subset J', \mbox{\tiny type 1}} e^{\phi_m(J)}\\
&\le \frac1{C_0} \
e^{(m-m')(\sup \phi - \inf \phi - \htop(f) + \gamma + \frac{\log 2}{R})}
e^{R' (\htop(f) - \gamma)}
\sum_{J \subset J', \mbox{\tiny type 1}} e^{\phi_m(J)}\\
&\le  \frac1{C_0} \ e^{R' \htop(f)}
\sum_{J \subset J', \mbox{\tiny type 1}} e^{\phi_m(J)}.
\end{align*}
Summing over all $m'$ and $J' \in \P_{m'}$, we get
\[
Z_m^{\mbox{\tiny type 2}}(\phi) \le \frac1{C_0} \ e^{R' \htop(f)}
Z_m^{\mbox{\tiny type 1}}(\phi).
\]
Now we combine this with \eqref{eq:type1} and the
fact that $\{ Z_n^{top}(\phi) \}_n$ is submultiplicative to obtain
\begin{align*}
e^{n P_{top}(\phi)} &\le Z_n^{top}(\phi) \le Z_{R'}^{top}(\phi) \cdot
Z_m^{top}(\phi)\le Z_{R'}^{top}(\phi)
\left[Z_m^{\mbox{\tiny type 1}}(\phi) + Z_m^{\mbox{\tiny type 2}}(\phi)\right] \\
&\le Z_{R'}^{top}(\phi) \left[1+ \frac1{C_0} e^{R' \htop(f)}\right]
Z_m^{\mbox{\tiny type 1}}(\phi)\\
&\le  Z_{R'}^{top}(\phi)
\left[1+ \frac1{C_0} e^{R' \htop(f)}\right] e^{-R' \inf \phi} e^{\beta_m} Z_n(\phi) \\
&\le  Z_{R'}^{top}(\phi)
\left(\frac2{C_0}\right) e^{R' (\htop(f)- \inf \phi)} e^{\beta_m - \beta_n}
e^{\beta_n} Z_n(\phi) = \frac1{\eta} e^{\beta_n} Z_n(\phi)
\end{align*}
for
$\eta = \left(\frac{C_0}{2 Z_{R'}^{top}(\phi) }\right)  e^{-R'(\htop(f)- \inf \phi)}
e^{\beta_n - \beta_m}$.
Since $n-m = R'$, we can assume that $e^{\beta_m - \beta_n}$
is bounded independently of $m$, so $\eta > 0$.
This proves the first statement.
In fact, since $\beta_n = o(n)$, we also find $P_{top}(\phi) = P_G(\phi)$.
\end{proof}

\begin{corollary}\label{cor:recurrence}
If $\sup \phi - \inf \phi < \htop(f)$ and
$\sum_n e^{-\beta_n} = \infty$, then the potential $\phi$
is recurrent.
\end{corollary}

\begin{proof}
Since $\phi$ is recurrent by definition if $\sum_n \lambda^{-n}
Z_n(\phi) = \infty$ for $\lambda = e^{P_G(\phi)}$, this corollary
is immediate from Lemma~\ref{lem:Zn}.
\end{proof}

The above ideas lead us to show that in our setting $P_{top}$ and $P_G$ are in fact the same.

\begin{corollary}\label{cor:PGhatphi}
If $\sup \phi - \inf \phi < \htop(f)$, then
$P_{top}(\hat\phi) = P_G(\hat\phi , \hat \c)$
for every cylinder set $\hat\c$ in $\hat I_{\mbox{\tiny trans}}$.
\end{corollary}

\begin{proof}
This is  the same proof as Lemma~\ref{lem:Zn}
with  $J \in \P_m$ replaced by $\hat J \in \hat\P_m \cap \hat \c$.
\end{proof}

\medskip
\noindent
Department of Mathematics\\
University of Surrey\\
Guildford, Surrey, GU2 7XH\\
UK\\
\texttt{h.bruin@surrey.ac.uk}\\
\texttt{http://www.maths.surrey.ac.uk/}
\\[0.5cm]
Department of Mathematics\\
University of Surrey\\
Guildford, Surrey, GU2 7XH\\
UK\footnote{
{\bf Current address:}\\
Departamento de Matem\'atica Pura\\
Faculdade de Ci\^encias da Universidade do Porto\\
Rua do Campo Alegre, 687\\
4169-007 Porto\\
Portugal\\
}\\
\texttt{mtodd@fc.up.pt}\\
\texttt{http://www.fc.up.pt/pessoas/mtodd/}

\end{document}